\documentclass
{amsproc}
\usepackage{amsmath}
\usepackage{amsthm}
\usepackage{amssymb}
\usepackage{graphicx}
\usepackage{multirow}
\usepackage{caption}
\usepackage{subfig}

\DeclareGraphicsExtensions{.png,.pdf,.eps}

\usepackage[all,2cell]{xy}
\usepackage{tikz}
\usepackage{pgf}
\usepackage{enumerate}
\usepackage{enumitem}
\usetikzlibrary{cd}

\usepackage[textsize=tiny]{todonotes}

\usepackage{array}
\newcolumntype{H}{>{\setbox0=\hbox\bgroup}c<{\egroup}@{}}
\newcolumntype{L}{>{$}l<{$}} 
\newcolumntype{C}{>{$}c<{$}} 

\usepackage{hyperref}
\DeclareMathOperator{\lcm}{lcm}

\newtheorem*{theorem*}{Theorem}
\newtheorem*{corollary*}{Corollary}
\newtheorem{theorem}{Theorem}[section]
\newtheorem{lemma}[theorem]{Lemma}

\newtheorem{proposition}[theorem]{Proposition}

\theoremstyle{definition}

\newtheorem{definition}[theorem]{Definition}
\newtheorem{notation}[theorem]{Notation}
\newtheorem{setup}[theorem]{Setup}

\newtheorem{remark}[theorem]{Remark}

\newtheorem{set}[theorem]{Setup}

\title[Birational transformations and $\C^*$-actions]{Geometric realizations of birational \\ transformations via $\C^*$-actions}

\author[Occhetta]{Gianluca Occhetta}
\address{Dipartimento di Matematica, Universit\`a degli Studi di Trento, via
Sommarive 14 I-38123 Povo di Trento (TN), Italy}
 
\email{gianluca.occhetta@unitn.it, eduardo.solaconde@unitn.it}

\author[Romano]{Eleonora A. Romano}
\address{Dipartimento di Matematica, Universit\`a degli Studi di Genova, via Dodecaneso 35, I-16146, Genova (GE), Italy}
\email{eleonoraanna.romano@unige.it}

\author[Sol\'a Conde]{Luis E. Sol\'a Conde}

\author[Wi\'sniewski]{Jaros\l{}aw A. Wi\'sniewski}
\address{Instytut Matematyki UW, Banacha 2, 02-097 Warszawa, Poland}
\email{J.Wisniewski@uw.edu.pl}


\thanks{First and third author supported by PRIN project ``Geometria delle variet\`a algebriche''. Fourth author supported by Polish National Science Center project 2016/23/G/ST1/04282.}

\DeclareMathOperator{\HH}{H}

\def\lcm{\operatorname{lcm}}

\newcommand\CC{{\mathbb{C}}}
\newcommand\PP{{\mathbb{P}}}

\def\C{{\mathbb C}}

\def\P{{\mathbb P}}
\def\Q{{\mathbb Q}}
\def\R{{\mathbb R}}
\def\Z{{\mathbb Z}}

\def\cC{{\mathcal C}}

\def\cI{{\mathcal I}}

\def\cN{{\mathcal{N}}}
\def\cO{{\mathcal{O}}}

\def\cU{{\mathcal U}}
\def\cV{{\mathcal V}}

\def\cY{{\mathcal Y}}
\def\Q{{\mathbb{Q}}}

\def\operatorname#1{\mathop{\rm #1}\nolimits}

\def\DA{{\rm A}}
\def\DB{{\rm B}}
\def\DC{{\rm C}}
\def\DD{{\rm D}}
\def\DE{{\rm E}}

\def\Proj{\operatorname{Proj}}

\def\Exc{\operatorname{Exc}}

\def\Hom{\operatorname{Hom}}

\def\Pic{\operatorname{Pic}}
\def\Hom{\operatorname{Hom}}

\def\rk{\operatorname{rk}}

\def\NE{{\operatorname{NE}}}

\def\Nef{{\operatorname{Nef}}}

\def\Mov{{\operatorname{Mov}}}

\def\SL{\operatorname{SL}}

\def\SP{\operatorname{Sp}}

\def\GX{\mathcal{G}\!X}

\newcommand{\pb}{\ar@{}[dr]|{\text{\pigpenfont J}}}
\def\ol{\overline}

\makeatletter
\newcommand{\xleftrightarrow}[2][]{\ext@arrow 3359\leftrightarrowfill@{#1}{#2}}
\makeatother
\newcommand{\xdasharrow}[2][->]{
\tikz[baseline=-\the\dimexpr\fontdimen22\textfont2\relax]{
\node[anchor=south,font=\scriptsize, inner ysep=1.5pt,outer xsep=2.2pt](x){#2};
\draw[shorten <=3.4pt,shorten >=3.4pt,dashed,#1](x.south west)--(x.south east);
}}

\newcommand\m{{\mathfrak m}}

\newcommand\lra{\longrightarrow}

\def\Mo{\operatorname{\hspace{0cm}M}}

\begin{document}
\begin{abstract}
In this paper we study varieties admitting torus actions as geometric realizations of birational transformations. We present an explicit construction of these geometric realizations for a particular class of birational transformations, and study some of their geometric properties, such as their Mori, Nef and Movable cones. 
\end{abstract}
\maketitle
\tableofcontents

\section{Introduction}\label{sec:intro}

Birational transformations of the projective space are called Cremona transformations, in honor of Luigi Cremona, who studied them in the 1860s (cf. \cite{Cre1,Cre2}). Among them, the 
standard Cremona transformation of the plane
$$\mathbb{P}^2\ni [x_0,x_1,x_2]\dashrightarrow [x_1x_2,x_0x_2,x_0x_1]\in \mathbb{P}^2$$
was apparently known to Pl{\"u}cker and Magnus in 1830s, as indicated in Coble's survey written $100$ years ago (see \cite{Coble} and references therein). 

In the 1990s Thaddeus and Reid linked birational geometry and the rising Minimal Model Theory to Mumford's Geometric Invariant Theory from the 1960s (GIT, \cite{MFK}), which was based on Hilbert's fundamentals from 1890s, see \cite{Thaddeus}. Subsequently, the concept of the algebraic cobordism was introduced by Morelli in the toric case, and by W{\l}odarczyk in full generality in the late 1990s \cite{Wlodarczyk}. The highlight of this approach was a proof of the factorization conjecture for birational maps of smooth projective varieties by Abramovich et al, 
\cite{Wlodarczyk-et-al}.

In the present paper we adopt a view on birational maps similar to the one of Morelli and W{\l}odarczyk. Namely, we describe a birational map among complex projective varieties of dimension $n$ (in particular a Cremona transformation) in terms of an algebraic action of the group $\mathbb{C}^*$ on a (projective) variety of dimension $n+1$; such a variety will be called a {\em geometric realization of the birational map}. For example, let us consider the action of $\mathbb{C}^*$ with coordinate $t$ on the product of three copies of $\mathbb{P}^1=\mathbb{C}\cup\{\infty\}$, each with (inhomogenenous) coordinate $y_i$: $$\mathbb{C}^*\times(\mathbb{P}^1\times\mathbb{P}^1\times\mathbb{P}^1)\ni (t,(y_1,y_2,y_3))\longrightarrow 
 (ty_1,ty_2,ty_3)\in\mathbb{P}^1\times\mathbb{P}^1\times\mathbb{P}^1$$
There exist $8$ fixed points of this action which are the triples $(y_1,y_2,y_3)$ with $y_i=0$ or $\infty$ and one can illustrate the action with the following diagram, representing the fixed points and the shortest nontrivial orbits (for clarity we do not label every fixed point).
\begin{center}
  \begin{tikzpicture}[scale=1]
    \draw[thin,fill=red!20] (4.7,0)--(5.1,-0.9)--(5.1,0.9)--cycle;
    \draw[thin,fill=blue!20] (1.4,0)--(1,-1)--(1,1)--cycle;
    \draw[thin,fill=black!10] %
    (3,-2)--(3.2,-1.2)--(3.2,1.2)--(3,2)--(2.8,1.2)--(2.8,-1.2)--cycle;
    \node[label=right:{$(0,0,0)$}] (a) at (6,0) {$\bullet$};
    \node[label=right:{$(\infty,0,0)$}] (b1) at (4,-2) {$\bullet$};
    \node
    (b2) at (4,0) {$\bullet$};
    \node[label=right:{$(0,0,\infty)$}] (b3) at (4,2) {$\bullet$};
    \node[label=left:{$(\infty,\infty,0)$}] (c1) at (2,-2) {$\bullet$};
    \node
    (c2) at (2,0) {$\bullet$};
    \node[label=left:{$(0,\infty,\infty)$}] (c3) at (2,2) {$\bullet$};
    \node[label=left:{$(\infty,\infty,\infty)$}] (d) at (0,0) {$\bullet$};
    \draw[thick,->] (a)--(b1); \draw[thick,->] (a)--(b2);  \draw[thick,->] (a)--(b3);
   \draw[thick,->] (b1)--(c1); \draw[thick,->] (b2)--(c1);  \draw[thick,->] (b2)--(c3);  \draw[thick,->] (b3)--(c3);
   \draw[thick,dashed, ->] (b1)--(c2); \draw[thick,dashed,->] (c2)--(d);  \draw[thick,dashed,->] (b3)--(c2);
   \draw[thick,->] (c1)--(d); \draw[thick,->] (c3)--(d);
  \end{tikzpicture}
\end{center}
The ``flow" of the action in this diagram is from right to left; the rightmost point is called the source and the leftmost point the sink of the action. The three vertical shaded sections of the diagram represent three GIT quotients, depending on the stability conditions. The triangular section on the right-hand side represents $\mathbb{P}^2$ parametrizing orbits diverging from the source, hence with homogeneous coordinates $[x_1:x_2:x_3]$. The section on the left-hand side represents $\mathbb{P}^2$ parametrizing orbits converging at $\infty$ to the sink, hence with homogeneous coordinates $[x_0^{-1}:x_1^{-1}:x_2^{-1}]$. If a general orbit has coordinates $[x_0:x_1:x_2]$ close to the source then its coordinates close to the sink are $[x_0^{-1}:x_1^{-1}:x_2^{-1}]=[x_1x_2:x_0x_2:x_0x_1]$.
We say that the variety $\P^1\times \P^1\times\P^1$ endowed with the $\C^*$-action introduced above is a geometric realization of the standard Cremona transformation.

In addition, the middle hexagonal section describes the intermediate (geometric) GIT quotient. This is the blowup of $\mathbb{P}^2$ in three points, that resolves the Cremona map; that is, it can be blown down to each of the two $\mathbb{P}^2$'s. In higher dimensions the blowdowns have to be replaced by more complicated birational operations like flips. Yet, as a general principle, presentation of a birational map by variation of quotients of a suitably chosen $\mathbb{C}^*$-action allows us to decompose the map as a sequence of simpler birational transformations. Based on this principle, W{\l}odarczyk proved his decomposition theorem.

Let us note that, if a birational map admits a geometric realization, then, universal objects for GIT theory, such as Chow or Hilbert quotients, if they are smooth, should provide a strong factorization of the map. This idea has been used in \cite{MMW} to resolve the birational map determined by the inversion of matrices.  

A natural question that appears in this setting, that is the main motivation of this paper, is how to construct {\em explicit} geometric realizations of birational maps. The problem has been already tackled in \cite{WORS1}, where the question has been solved for Atiyah flips \cite[Section~6]{WORS1} and for special quadro-quadric Cremona transformations (see Section \ref{sec:BW3} below). In this paper we go one step further in this direction by constructing smooth geometric realizations for bispecial transformations, a class of birational maps that contains special quadro-quadric Cremona transformations.

\medskip
\noindent{\bf Outline.} In the first part of the paper we  present background material on $\mathbb{C}^*$-actions and describe the birational transformations they define (Section \ref{sec:prelim}). In particular, in Section \ref{sec:defgeomreal} we introduce the central subject of the paper: namely, the definition of geometric realization of a birational transformation. In Section \ref{sec:BW3} we recall the classification of equalized actions of bandwidth three with isolated sink and source, and the relation with special quadro--quadric Cremona transformations. This classification leads to the following statement (Theorem \ref{thm:bw3}, see also Notation \ref{not:RH} for conventions regarding rational homogeneous varieties):

\begin{theorem*}\label{thm*:bw3}
Any quadro-quadric Cremona transformation with smooth nonempty fundamental locus admits a geometric realization, given by an equalized $\C^*$-action of criticality three in one of the following varieties:
$$
\P^1\times Q^{n-1},\qquad \DC_3(3),\qquad \DA_5(3),\qquad  \DD_6(6),\qquad \DE_7(7).
$$
\end{theorem*}

The last two sections are devoted to the construction of geometric realizations of bispecial transformations. 
First (Section \ref{sec:geomreal}) we 
show how to construct geometric realization in the case in which the transformation is bispecial between smooth projective varieties. The main result of this section is the following:
\begin{theorem*}\label{thm*:bispecial}
Let $\psi$
be a bispecial birational transformation between smooth polarized varieties. Then there exists a  smooth B-type equalized geometric realization of $\psi$ of criticality three. 
\end{theorem*}

Furthermore, we show that such a geometric realization is unique under some additional assumptions (see Theorem \ref{thm:gr1} for the precise statement). 
 
Then, in Section \ref{sec:cones}, we study some geometric properties of these realizations, encoded in their Nef, Mori and Movable cones (Propositions \ref{prop:cones}, \ref{prop:movx}). For simplicity we focus on the case of birational transformations between varieties of Picard number one. In particular we discuss when the sink and the source of the geometric realizations may be equivariantly contracted to points, and when the (non B-type) geometric realization of $\psi$ obtained in this way is smooth (Theorems \ref{thm:bispecial}, \ref{thm:bispecial2}).

\section{ $\C^*$-actions}\label{sec:prelim}

This section contains some  results, notation and conventions regarding $\C^*$-actions and the birational transformations that they induce. We refer the interested reader to \cite{BB,CARRELL,WORS1,WORS3,B_R} for details and further results.


Let $X$ be an irreducible complex algebraic variety. By {\em $\C^*$-action on $X$} we will always refer to a morphical  $\C^*$-action on $X$, that is an action of the group $\C^*$ on $X$, whose defining map
$$\C^*\times X\to X$$
is a morphism of algebraic varieties. The action of $t\in\C^*$ on $x\in X$ will be always denoted by $tx$. We will usually consider $X$ to be proper.

\begin{remark}\label{rem:extraction}
We will use later on the fact that if $Y\subset X$ is a $\C^*$-invariant closed subset of $X$, then the action  extends to the blowup of $X$ along $Y$:
$$
\beta:X^\flat\to X.
$$
This holds also for the projectivization of  the $\cO_X$-algebra $\bigoplus_{m\geq 0}\cI_{Y}^m$ with respect to a grading different from the natural one. In other words, the $\C^*$-action on $X$ extends to any weighted blowup of $X$ along a $\C^*$-invariant closed subset.
\end{remark}

\subsection{Fixed-point components}\label{sssec:fxd}
We will denote by $X^{\C^*}$ the set of fixed components of the action, whose union is a closed set in $X$, and by $\cY$ the set of its irreducible  components. If $X$ is proper, then we may consider, for a general point $x\in X$, the limiting points
$$
\lim_{t\to 0}t^{-1}x, \quad \lim_{t\to 0}tx \in X,
$$
called, respectively, the {\em sink} and the {\em source} of the orbit $\C^*x$. 
The components $Y_-,Y_+$ containing them are called, respectively, the {\em sink} and the {\em source of the action}. The rest of the fixed-point components are called {\em inner}, and their set is denoted by $\cY^\circ$.
Moreover, if $X$ is smooth and projective, every $Y\in\cY$ is smooth (cf. \cite[Main Theorem]{IVERSEN}).

\subsection{Linearizations and weight maps}\label{sssec:lin} Given a line bundle $L$ on $X$, a {\em linearization}  on $L$ of the $\C^*$-action on $X$ is an action of $\C^*$ on $L$ such that the projection map $L\to X$ is equivariant and the action is linear on the fibers of $L\to X$. If $X$ is normal and projective, linearizations exist for any line bundle  on $X$ (see \cite[Proposition 2.4]{KKLV}, and the Remark right after it).  Given a fixed-point component $Y\in \cY$, $\C^*$ acts on $L_{|Y}$ by multiplication with a character $\mu_L(Y)\in \Mo(\C^*)=\Hom(\C^*,\C^*)=\Z$, which is called {\em weight of the linearization on $Y$}.
The map $\mu_L:\cY\to\Z$ obviously depends on the linearization on $L$, but one may easily prove that two linearizations on a line bundle $L$ differ by a character of $\C^*$; then we may always choose the linearization so that $\mu_L$ has a prescribed value on a certain fixed-point component. Usually we will choose the linearization so that $\mu_L(Y_-)=0$, and call $\mu_L(Y)$ the {\em $L$-weight of the action on $Y$}, for every $Y\in \cY$. Note that with this convention we may write $\mu_{mL}=m\mu_L$ for every $m\in \Z$.

If $L$ is ample the minimum and maximum value of the weights are achieved at the sink and the source of the action, respectively (see \cite[Remark 2.12]{WORS1}). In particular, the difference $\delta:=\mu_L(Y_+)-\mu_L(Y_-)$, that is called the {\em bandwidth of the $\C^*$-action on $(X,L)$}, is independent of the chosen linearization. 

\subsection{Actions on polarized pairs. Criticality}\label{sssec:polar} A $\C^*$-action on a projective variety $X$, together with a linearization on an ample line bundle $L$ on $X$ will be referred to as a {\em $\C^*$-action on the 
polarized pair $(X,L)$}. Denoting by 
$$
\{a_0=0,\dots,a_r=:\delta\} \quad\mbox{(with $a_0<\dots<a_r$)}$$
the set $\mu_L(\cY)\subset \Z$ of $L$-weights of the action, we set:
$$
Y_i:=\bigcup_{\mu_L(Y)=a_i}Y\subset X.
$$
The number $r$, which does not change if we substitute $L$ by a multiple, is called the {\em criticality of the $\C^*$-action} on $(X,L)$. The minimum of the criticalities of the $\C^*$-actions on the pairs $(X,L)$, when $L$ varies in the set of ample divisors on $X$, will be called the {\em criticality} of the $\C^*$-action on $X$.

\subsection{The Bia{\l}ynicki-Birula decomposition}\label{sssec:BBcells} We refer to \cite{CARRELL} for a complete account on the Bia{\l}ynicki-Birula decomposition and its applications, and to \cite{BB} for the original reference; let us describe here the contents of this decomposition theorem. Assume that $X$ is a proper variety admitting a $\C^*$-action as above.
Given any set $S \subset X$, we consider:
\begin{equation}\label{eq:BBcells}
X^\pm(S):=\{x\in X|\,\, \lim_{t^{\pm 1}\to 0} tx\in S\},
\end{equation}
When $S=Y$ with $Y\in \cY$ we call $X^\pm(Y)$ the positive and the negative {\it Bia{\l}ynicki-Birula cells} of the action at $Y$, so that we have a decomposition: 
$$X=\bigsqcup_{Y\in\cY}X^+(Y)=\bigsqcup_{Y\in\cY}X^-(Y).$$ 

If $X$ is smooth, $\C^*$ acts (fiberwise linearly) on the normal bundle $\cN_{Y|X}$ of $Y$ in $X$. The weights of the action on the fibers of the bundle are the same for every point of $Y$, and $\cN_{Y|X}$ splits into two subbundles, on which $\C^*$ acts with positive and negative weights, respectively:
\begin{equation}
\cN_{Y|X}\simeq \cN^+(Y)\oplus \cN^-(Y).
\label{eq:normal+-}
\end{equation}
Then the action of $\C^*$ on $X^{\pm}(Y)$ is equivariantly isomorphic to the induced action on the bundles $\cN^{\pm}(Y)$.

\begin{remark}\label{rem:innernu}
If $X$ is smooth, we will set, for every $Y\in\cY$
\begin{equation} \label{eq:Vpm}
\nu^{\pm}(Y):=\rk \cN^{\pm}(Y), 
\end{equation}
so that from the Bia{\l}ynicki-Birula decomposition it follows that:
\begin{equation} \label{eq:Vpm2}
\dim(X)=
\dim(\ol{X^+(Y)})+\nu^-(Y)=\dim(\ol{X^-(Y)})+\nu^+(Y)\end{equation}
for any $Y\in \cY$. 
In particular, this implies that $\nu^\pm(Y)\geq 1$ for every inner fixed-point component $Y\in \cY^\circ$; in fact, if, for instance, $\nu^+(Y)=0$ for some $Y\in \cY$, then (\ref{eq:Vpm2}) says that $X^-(Y)$ is dense in $X$, and so $Y$ is the sink of the action. 

\end{remark}

\subsection{Geometric quotients}\label{ssec:quot}

Recall that a linearization of the $\CC^*$-action on a line bundle $L$ over $X$ provides a weight decomposition on $\HH^0(X,L)$: 
$$\HH^0(X,L)=\bigoplus_{u\in\Z}\HH^0(X,L)_u,$$
where $\HH^0(X,L)_u\subset \HH^0(X,L)$ stands for the vector subspace of $\C^*$-weight $u$.
If $L$ is ample and globally generated, one may show (cf. \cite[Lemma~2.4 (3)]{BWW}) that the extremal values for which $\HH^0(X,L)_u\neq \{0\}$ are $\mu_L(Y_-)=0$, $\mu_L(Y_+)=\delta$. In fact, the quoted statement shows that the convex hull $\Delta(L)$ of the weights $\mu_L(Y)$, $Y\in \cY$ (which in our case is $[0,\delta]$) coincides with the convex hull $\Gamma(L)$ of the weights of the induced action on $\HH^0(X,L)$. 

Let us now recall how to construct algebraically the geometric (and semigeometric quotients) of the polarized pair $(X,L)$; we refer the interested reader to \cite[Section~2.2]{WORS3} and \cite{BBS2} for details. Given a $\C^*$-action on a pair $(X,L)$, with $\mu_L(\cY)=\{a_0=0,\dots,a_r=\delta\}$, and given any rational number $\tau\in[0,\delta]\cap \Q$ we may consider the graded $\C$-algebra 
$$
R_{L,\tau}:=\bigoplus_{\substack{m\geq 0\\m\tau\in\Z}}\HH^0(X,mL)_{m\tau}.
$$
It is known that the projectivization of $R_{L,\tau}$ is a GIT quotient of an open set $X^s\subset X$ of stable points by the action of $\C^*$. More precisely, $\Proj(R_{L,\tau})$ is a geometric quotient if $\tau\neq a_i$ for every $i$, and a semigeometric quotient otherwise. Furthermore, given two rational numbers $\tau,\tau'\in(a_{i-1},a_i)$ for some $i$, the varieties $\Proj(R_{L,\tau}),\Proj(R_{L,\tau'})$ are isomorphic, and it then makes sense to define, for every $i\in{1,\dots,r}$:
$$\GX_i:=\Proj(R_{L,\tau_i}),\quad i=1,\dots,r,$$
where $\tau_i$ is any rational number in $(a_{i-1},a_i)$; the corresponding open set of stable base points is denoted $X^s_i\subset X$. We call the varieties $\GX_1,\dots,\GX_r$ the {\em geometric quotients of the polarized pair} $(X,L)$. 

\begin{remark}\label{rem:extrquot}
Following \cite{BBS2}, the open sets $X^s_i\subset X$ can be described in terms of sections of the ordered set of fixed-point components $\cY$; in this way one may see for instance that $\GX_1,\GX_r$ parametrize the $1$-dimensional orbits converging, respectively, to the sink and the source, that is  $X^-(Y_-)\setminus Y_-$, $X^+(Y_+)\setminus Y_+$. In particular, these two quotients do not depend on the particular polarization $L$, but only on the $\C^*$-action on $X$.
\end{remark}

\subsection{Equalized actions}\label{sssec:equal}

A $\C^*$-action on a proper variety $X$ is called {\em  equalized} at $Y\in\cY$ if for every $x\in \big(X^{-}(Y)\cup X^{+}(Y)\big)\setminus Y$  the isotropy group of the $\C^*$-action on $x$ is trivial, and is called {\em equalized} if it is equalized at every fixed-point component. By the Bia{\l}ynicki-Birula Theorem, if the variety $X$ is smooth, the equalization of the action is equivalent to the weights of the action on $\cN^{\pm}(Y)$ being all equal to $\pm 1$ for every fixed-point component $Y\in\cY$ (as presented in \cite[Definition~1]{RW}).

Note that if a $\C^*$-action is equalized, then the closure of any $1$-dim\-ens\-ional orbit is a smooth rational curve (see \cite[Corollary 2.9]{WORS1}), whose $L$-degree may be then computed in terms of the weights at its extremal points. The precise relation among these values has been introduced in \cite{RW} as the AM vs. FM formula, that we present here only for equalized actions (see \cite[Corollary 3.2 (c)]{RW}):

\begin{lemma}[AM vs. FM]\label{lem:AMFM}
Let $X$ be a smooth projective variety admitting a nontrivial equalized $\C^*$-action, $L$ an ample line bundle on $X$, and let $C$ be the closure of a $1$-dimensional orbit, whose sink and source are denoted by $x_-$ and $x_+$. Then $C$ is a smooth rational curve of $L$-degree  equal to $\mu_L(x_+)-\mu_L(x_-)$.
\end{lemma}

\subsection{B-type torus actions}\label{sssec:Btype}

\begin{definition}\label{def:bord}
Let  $X$ be a smooth projective variety admitting a $\C^*$-action. We say that the action is of {\em B-type} if its sink and source, $Y_-,Y_+$, are codimension one subvarieties.
\end{definition}

If we start from a non B-type $\C^*$-action, we may perform first a $\C^*$-equivariant divisorial extraction (Remark \ref{rem:extraction}) centered in the extremal fixed-point components $Y_{\pm}$ in order to get a B-type action, with the same $1$-dimensional orbits as in the original variety. In the equalized setting one can prove  (cf. \cite[Lemma~3.10]{WORS1}): 

\begin{lemma}\label{lem:blowup}
Given a nontrivial equalized $\C^*$-action on a smooth projective variety $X$ with sink and source $Y_\pm$, the induced  $\C^*$-action  on the blowup of $X$ along $Y_{\pm}$ is of B-type.
\end{lemma}

\begin{remark}\label{rem:weighted}
The blowup operation substitutes the extremal fixed-point components $Y_{\pm}$ with the projective bundles $\P(\cN_{Y_{\pm}|X})$, which, by the Bia{\l}ynicki-Birula Theorem, are isomorphic to the  quotients of $X^{\pm}(Y_{\pm})\setminus Y_{\pm}$ by the action of $\C^*$, that is, the geometric quotients $\GX_1,\GX_r$ introduced in Section \ref{ssec:quot}. If the action is not equalized, then a similar construction works if one uses  the weighted blowup with respect to the weights of the $\C^*$-action on $\cN_{Y_{\pm}|X}$ instead of the standard blowup; see \cite[$\S$3.2 and Lemma 4.4]{B_R} for details. 
\end{remark}

Let us finally note the following:

\begin{remark}\label{rem:Btypeequalized}
If a B-type action is equalized, then the closure $C$ of any orbit in $X$ satisfies that $C\cdot Y_{\pm}=1$ (see \cite[Remark~3.3]{WORS1}). 
\end{remark}

\subsection{Geometric realizations of birational maps}\label{sec:defgeomreal}

Given any nontrivial $\C^*$-action on a polarized pair $(X,L)$, with $X$ normal and projective, we consider the geometric quotients $\GX_i=X^s_i/\C^*$ defined in Section \ref{ssec:quot}. 
Given $i<r$, the intersection $X_i^s\cap X_{i+1}^s$ is a $\C^*$-invariant open set in $X$, whose images into $\GX_i$ and $\GX_{i+1}$ are nonempty open sets providing a birational map $\GX_i\dashrightarrow \GX_{i+1}$. Summing up, we have birational maps:  
$$
\xymatrix@C=14mm{\GX_1\ar@{-->}[r]& \GX_2\ar@{-->}[r]&\quad\dots\quad \ar@{-->}[r]& \GX_{r-1}\ar@{-->}[r]&\GX_r}
$$
In particular we have a birational map $\psi:\GX_1\dashrightarrow \GX_r$; we will call $\GX_1$ and $\GX_r$ the extremal geometric quotients. As in Remark \ref{rem:extrquot}, $\psi$ does not depend on the chosen ample line bundle $L$, and it is called the {\em birational map associated to the $\C^*$-action on $X$}.

\begin{remark}\label{rem:birmap}
In the case in which $X$ is smooth and the $\C^*$-action on $X$ is of B-type, 
then the Bia{\l}ynicki-Birula theorem tells us that the extremal geometric quotients of $X$ are precisely the sink and the source $Y_\pm$. Then, if we apply this to the blowup $X^\flat$ along the sink and the source of a smooth variety $X$ endowed with an equalized $\C^*$-action (see Lemma \ref{lem:blowup}) we obtain a birational map:
$$
\xymatrix{\psi:\P(\cN_{Y_-|X}^\vee)\ar@{-->}[rr]&&\P(\cN_{Y_+|X}^\vee).}
$$
As noted above (see \cite[Lemma 4.4]{B_R}), a similar approach is possible in the non-equalized case via weighted blowups.
\end{remark}

\begin{definition}\label{def:geomreal}
Given a birational map between two normal projective varieties $\psi:M_-\dashrightarrow M_+$, a {\em geometric realization of} $\psi$ is a normal projective variety endowed with a nontrivial $\C^*$-action, whose extremal geometric quotients are isomorphic to $M_\pm$, and whose associated birational map is $\psi$.
\end{definition}

\begin{remark}\label{rem:uniquegeorel}
One cannot expect birational maps to have unique geometric realizations, since performing birational $\C^*$-equivariant transformations on a geometric realization of a map, we obtain other realizations of the same map. For instance, assume that we start with a faithful $\C^*$-action on a smooth variety $X$ that has an inner fixed-point component $Y$. The $\C^*$-action on $X$ naturally extends to the blowup $X^\flat$ of $X$ along $Y$, and the birational maps associated to the actions of $\C^*$ on $X$ and $X^\flat$ are obviously the same. 
\end{remark}

\section{Equalized $\C^*$-actions with isolated extremal fixed points}\label{sec:BW3}

In this section we review the  classification theorem for bandwidth three varieties admitting an equalized $\C^*$-action with isolated extremal fixed points. This result was firstly motivated in \cite{RW} by the LeBrun--Salamon conjecture. In fact, bandwidth three varieties appear naturally as subvarieties of some Fano contact manifolds (see also \cite[proof of Theorem 5.1]{WORS2}). The important point that we want to stress  is the tight relation of a $\C^*$-action with the birational transformation induced by it, which has been fundamental in the proof of this result. We will explore deeply this idea in the following sections.

Before stating the main result of this Section, let us introduce the notation regarding the rational homogeneous varieties that will appear later on. Our notation is compatible with our main source for this Section, \cite{WORS1}, and essentially presents every rational homogeneous variety as the closed orbit of a projective representation of a certain semisimple group.

\begin{notation}\label{not:RH}
In the following description we use the numbering of fundamental weights provided in \cite[Planche I--IX]{Bourb}.
\begin{itemize}[leftmargin=*]
\item $\DA_5(2),\DA_5(3)$ denote, respectively, the Grassmannians of $2$ and $3$-dimensional linear subspaces in a complex vector space of dimension $6$. They are the closed orbits of the projective representations of $\SL(6)$ (whose Lie algebra is determined by the Dynkin diagram $\DA_5$) given by the fundamental weights $\omega_2,\omega_3$. 
\item $\DC_3(3)$ denotes the Lagrangian Grassmannian parametrizing $3$-dimensional vector spaces in a complex vector space of dimension $6$ that are isotropic with respect a given nondegenerate skew-symmetric form. It is the closed orbit of the projective representation of $\SP(6)$ (whose Lie algebra is given by $\DC_3$) given by the fundamental weight $\omega_3$. 
\item $\DD_6(6)$ denotes the Spinor variety parametrizing $6$-dimensional vector spaces in a complex vector space of dimension $12$ that are isotropic with respect a given nondegenerate symmetric form. It is determined by the fundamental weight $\omega_6$ of the Lie algebra with Dynkin diagram $\DD_6$.
\item $\DE_6(1)$, called Cartan variety, denotes the closed orbit of the projective representation of the adjoint group of the Lie algebra $\DE_6$ given by its weight $\omega_6$. 
\item $\DE_7(7)$ denotes the closed orbit of the projective representation of the adjoint group of the Lie algebra $\DE_7$ given by its fundamental weight $\omega_7$.
\item With this notation, smooth quadrics of dimension $k=2r-1,2r-2$ can be written as $\DB_k(1),\DD_k(1)$, respectively. Since in our discussion below the parity of the dimension of the quadrics is irrelevant, we will simply denote them by $Q^k$.
\end{itemize}
\end{notation}

\begin{theorem}\label{thm:bw3} \cite[Theorem~4.5]{RW}, \cite[Theorem 8.1]{WORS1} 
  Let $(X,L)$ be a polarized pair, where $X$ is a  smooth projective variety of
  dimension $n\geq 3$, endowed with a $\C^*$-action of bandwidth three. Assume that its sink and source are isolated points, and that the action is equalized. 
Then one of the following holds:
  \begin{itemize}
 \item [(1)] $X=\P(\cV^\vee)$, 
 with $\cV=\cO_{\P^1}(1)^{\oplus n-1}\oplus\cO_{\P^1}(3)$, or $\cO_{\P^1}(1)^{\oplus{n-2}}\oplus\cO_{\P^1}(2)^{\oplus 2}$, and $L= \cO_{\P(\cV^\vee)}(1)$. Moreover  $(Y_i,L_{|Y_i}) \simeq (\P^{n-2}, \cO_{\P^{n-2}}(1))$, $i=1,2$.
\item [(2)] $X=\P^1\times\Q^{n-1}$, $L=\cO(1,1)$, each $Y_i$ is the disjoint union of a smooth quadric $\Q^{n-3}$ and a point, and $L_{|\Q^{n-3}} \simeq \cO_{\Q^{n-3}}(1)$.
  \item [(3)]  $X$ is one of the following rational homogeneous varieties:
$$\DC_3(3),\,\,\, \DA_5(3),\,\,\,  \DD_6(6),\,\,\, \DE_7(7),$$
$L$ is the ample generator of $\Pic(X)$ and the varieties $Y_i$ are respectively
$$\P^2,\,\,\, \P^2 \times \P^2,\,\,\,  \DA_5(2),\,\,\, \DE_6(1).$$
The restriction of $L$ to $Y_i$ is the ample generator of $\Pic(Y_i)$, except in the case $Y_i\simeq \P^2$, in which $L_{|Y_i} \simeq \cO_{\P^2}(2)$.
  \end{itemize}
\end{theorem}

The proof of the above result has been done adopting different techniques in \cite{RW} and \cite{WORS1}; we refer to \cite[Lemma 4.4]{RW} for the case $n=2$. More precisely, in \cite[Theorem~3.5]{RW} the authors relate the classical adjunction theory and Mori theory with a combinatorial description of $\C^*$-varieties; in this way types $(1)$ and $(2)$ of pairs $(X,L)$ in the above list are described in terms of their adjoint morphisms. 

The classification in the case in which the Picard number of $X$ is one (type $(3)$ in the statement) was later obtained in \cite[Theorem 8.1]{WORS1} by considering the birational map $\psi$ associated to the $\C^*$-action, 
which, since $Y_\pm$ are isolated points, is a Cremona transformation $\psi:\P^{n-1}\dashrightarrow\P^{n-1}$ by Remark \ref{rem:birmap}. Let us briefly discuss the different steps that lead to the proof of the statement in this case. 

\begin{proof}[Sketch of the proof ($\Pic(X)\simeq \Z)$]\quad\par

\noindent{\em Step 0.} Prove that, since $\Pic(X)\simeq \Z$, the two subvarieties $Y_1,Y_2$ are irreducible (cf. \cite[Lemma~2.8(2)]{WORS1}). \par\medskip

\noindent{\em Step 1.} Blow up $X$ along $Y_\pm$, to obtain a variety $X_\flat$ with a $\C^*$-action whose extremal fixed-point components are  projective spaces $\P^{n-1}_\pm$. Denote by $\cU_1$, $\cU_2$ the closure of the Bia{\l}ynicki-Birula cells $X_{\flat}^+(Y_1)$, $X_{\flat}^-(Y_2)$, respectively. \par\medskip

\noindent{\em Step 2.} Blow up $X_\flat$ along $\cU_1$, $\cU_2$, and let $\cU^\flat_1$, $\cU^\flat_2$ be the corresponding exceptional divisors. Denoting by $X'_{\flat}$ the resulting variety, show that there exists another divisorial contraction $X'_{\flat}\to X'$, whose restriction to the $\cU^\flat_i$'s is a nontrivial projective bundle different from the projection $\cU^\flat_i\to\cU_i$. See Figure \ref{fig:bw3} below. \par\medskip

\begin{figure}[h!!]
$$
\xymatrix@R=20pt{\includegraphics[width=5cm]{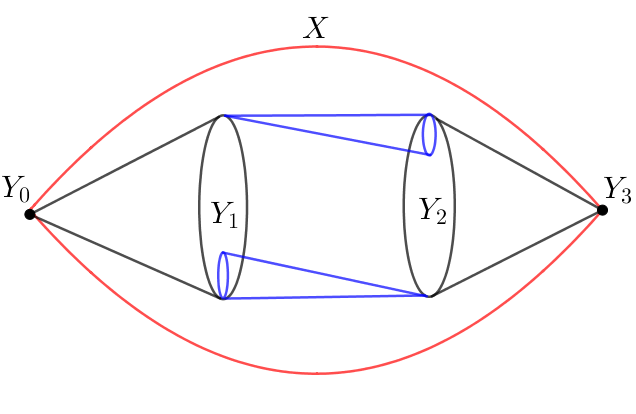}\ar@{-->}[d]&\includegraphics[width=4.5 cm]{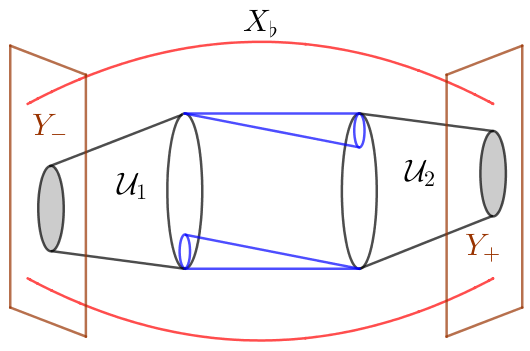}\ar[]!<0ex,10ex>;[l]!<0ex,10ex>_(.47){\varphi}\\
\includegraphics[width=5.4cm]{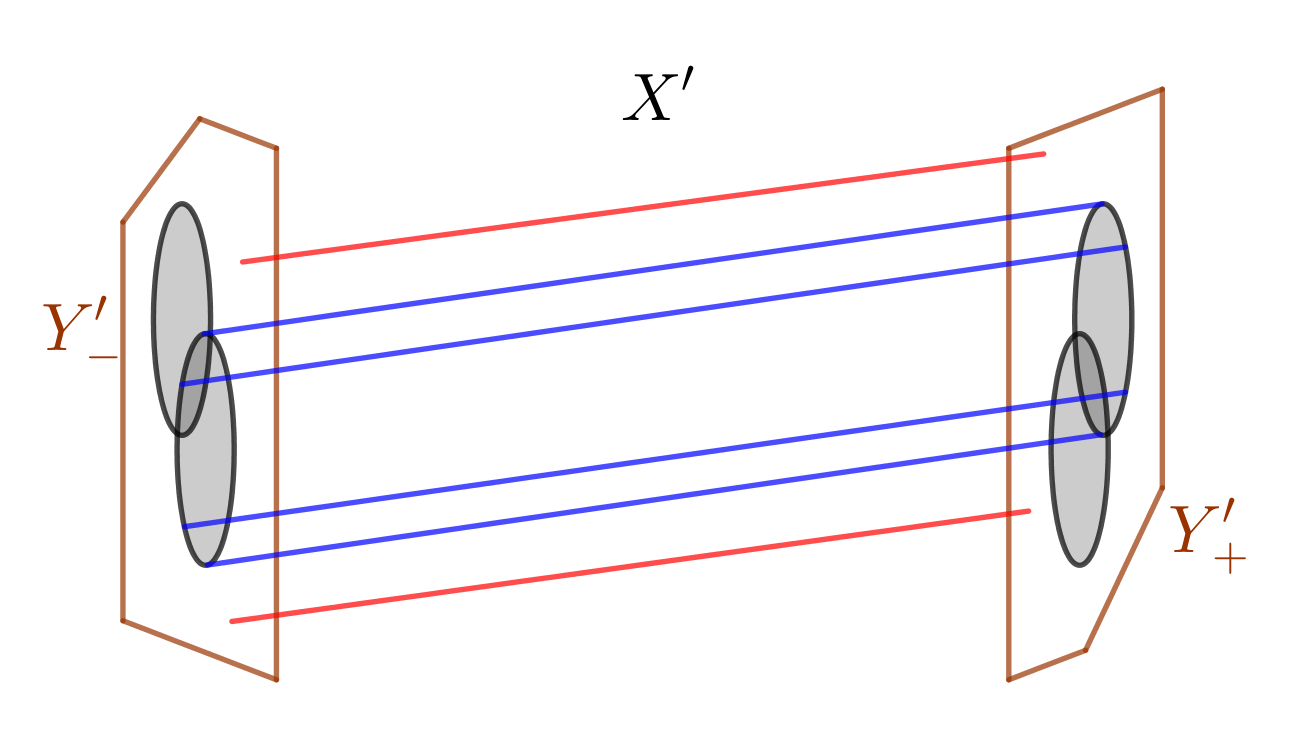}&\includegraphics[width=5cm]{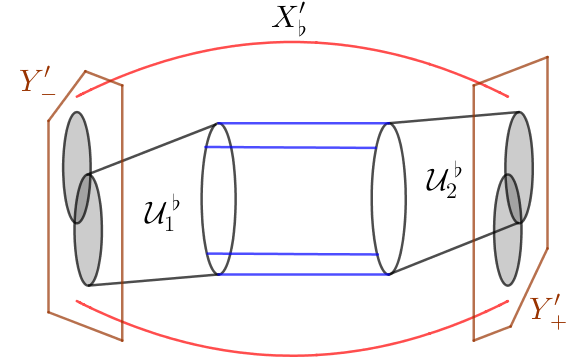}\ar[u]^(.85){\alpha}\ar[]!<0ex,10ex>;[l]!<0ex,10ex>_(.45){\varphi'}}
$$
\caption{Birational transformation of a bandwidth three variety with isolated extremal fixed points into a $\P^1$-bundle.
\label{fig:bw3}
}
\end{figure}

\noindent{\em Step 3.} The variety $X'$ inherits a $\C^*$-action  whose sink and source ($Y'_\pm$) are respectively isomorphic to the blowup of $\P^{n-1}_-$ along $\cU_1\cap \P^{n-1}_-\simeq Y_1$, and to the blowup of $\P^{n-1}_+$ along $\cU_2\cap \P^{n-1}_+\simeq Y_2$. Since the criticality of the induced $\C^*$-action on $X'$ is one (that is, it has no inner fixed point components), it follows that the associated birational map among $Y'_-$ and $Y'_+$ is an isomorphism. In other words, the map $\psi$ and its inverse can be resolved with a smooth blowup.
$$\xymatrix@R=8pt{&Y'_-\simeq Y'_+\ar[ld]_{\mbox{\tiny blowup}}\ar[rd]^{\mbox{\tiny blowup}}&\\\P^{n-1}_-&&\P^{n-1}_+}  
$$

\par\medskip

\noindent{\em Step 4.} An intersection theory computation (see the proof of \cite[Theorem~8.4]{WORS1}) shows that the bidegree of the birational map $\psi$ is $(2,2)$. \par\medskip 

\noindent{\em Step 5.}  Since $Y_1,Y_2$ (which are the indeterminacy loci of $\psi,\psi^{-1}$) are irreducible, we may apply \cite[Theorem~2.6]{ESB} to conclude that $\psi$ is one of the four quadro-quadric Cremona transformations defined by the linear system of quadrics containing a Severi variety:
$$v_2(\P^2)\subset \P^5,\quad \P^2\times\P^2\subset\P^8,\quad \DA_5(2)\subset\P^{14},\quad\DE_6(1)\subset \P^{26}.
$$

\noindent{\em Step 6.} Reverting the construction described in Steps 1, 2 and  3 one shows that $X$ is uniquely determined by $\psi$. We then conclude by proving that the varieties $$\DC_3(3),\,\,\, \DA_5(3),\,\,\,  \DD_6(6),\,\,\, \DE_7(7)$$ admit $\C^*$-actions whose birational transformations are the above ones. \end{proof}

\begin{remark}\label{rem:BW3}
The birational construction performed in Steps 1, 2 and 3 (represented in Figure \ref{fig:bw3}) does not require the irreducibility of $Y_1,Y_2$, and can be extended to arbitrary bandwidth (see  \cite[Theorem~3.1]{WORS3}). 
\end{remark}

Moreover, we may still claim that also in the cases of Picard number two of Theorem \ref{thm:bw3} the variety $X$ admitting an equalized $\C^*$-action  of bandwidth three with isolated extremal fixed points is uniquely determined by the subjacent special Cremona transformation $\psi$, and that $\psi$ has bidegree $(2,2)$. We then note that, by using completely different methods, Pirio and Russo (cf. \cite[Proposition 5.6]{PiRu}) have extended the classification of Ein and Shepherd-Barron to include the case in which the exceptional locus of $\psi$ is reducible. Their list includes Cremona transformations
$\P^{n-1}\dashrightarrow \P^{n-1}$ having as fundamental locus the union of a point and a smooth quadric $\Q^{n-3}$, which is precisely the birational transformation that we obtain in Theorem \ref{thm:bw3}(2). We finally note that in case (1) the subvariety $Y_1\subset\P^{n-1}_-$ 
is a codimension one linear subspace, so the linear system of quadrics defining $\psi$ has a fixed hyperplane; dividing by its equation we obtain that $\psi$ is a linear isomorphism. 
In other words, Theorem \ref{thm:bw3} implies the following:

\begin{theorem}\label{thm:bw3}
Any quadro-quadric Cremona transformation with nonempty smooth fundamental locus admits a geometric realization, given by an equalized $\C^*$-action in one of the following varieties:
$$
\P^1\times Q^{n-1},\qquad \DC_3(3),\qquad \DA_5(3),\qquad  \DD_6(6),\qquad \DE_7(7).
$$
\end{theorem}

\section{Constructing geometric realizations}\label{sec:geomreal}

We have seen how one can associate to a B-type $\C^*$-action a birational transformation between the sink and the source; in this section we want to show that the converse holds for a certain class of birational transformations among smooth projective varieties, that we will introduce herein. 

\begin{definition}\label{def:bispecial}
Let $(Y_-,H_-),(Y_+,H_+)$ two smooth polarized pairs of dimension $n$. A {\em bispecial transformation} between $(Y_-,H_-)$ and $(Y_+,H_+)$ is a birational map $\psi:Y_-\dashrightarrow Y_+$ that can be resolved via two blowups $\pi_\pm$ along smooth subvarieties $Z_{\pm}\subset Y_{\pm}$ such that, denoting by $E_{\pm}=\pi_{\pm}^{-1}(Z_{\pm})$ the corresponding exceptional divisors,  there exist integers $m_\pm> 0$ with:
\begin{equation}
\label{eq:cr1} \pi_+^*H_+\sim m_- \pi_-^*H_- -E_-, \qquad \pi_-^*H_- \sim m_+\pi_+^*H_+ -E_+
\end{equation}
We call $(m_-,m_+)$ the {\em type} of the bispecial transformation.
\begin{equation}\label{eq:diagbi}\xymatrix@R=30pt@C=20pt{ & E_- \ar[ld] \ar@{^(->}[r]& \ar[ld]_{\pi_-}W \ar[rd]^{\pi_+} & E_+ \ar[rd] \ar@{_(->}[l]& \\
Z_- \ar@{^(->}[r]&Y_{-} \ar@{-->}@/^10pt/[rr]^{\psi}& &Y_{+} \ar@{-->}@/^10pt/[ll]^{\psi^{-1}}&\ar@{_(->}[l] Z_+ }
\end{equation}
\end{definition}

Note that, from the definition it follows also that:
\[E_-+m_-E_+ \sim (m_-m_+-1)\pi_+^*H_+ ,\quad E_+ +m_+E_- \sim (m_-m_+-1)\pi_-^*H_- .\]
In particular, since $E_\pm$ are effective, we have that $m_-m_+>1$.

\begin{remark}\label{rem:classes}
In the case in which $\Pic(Y_\pm)$ have rank one, condition (\ref{eq:cr1}) implies that $H_\pm$ are the ample generators of the Picard groups. In fact, denoting by $[e_\pm]$ the numerical class of a nontrivial fiber of $\pi_\pm$, and by $[\ell_{\pm}]=\pi_{\mp*}[e_\pm]$ the class of its image into $Y_\mp$, it follows that:
$$
[\ell_{\mp}]\cdot H_\mp=[e_\pm]\cdot\pi_\mp^*H_\mp= [e_\pm]\cdot (m_\pm\pi_\pm^*H_\pm -E_\pm)=1.
$$
\end{remark}

\begin{remark}\label{rem:special}
In the classical setting in which $Y_\pm$ is the projective space, a birational transformation with smooth and irreducible fundamental locus is called a {\em special Cremona transformation} (see \cite{ESB}). In the definition of bi\-spe\-cial transformation we do not require $Y_\pm$ to be projective spaces, nor the fundamental locus to be irreducible; however we do require the fundamental loci of both the transformation and its inverse to be smooth. 
\end{remark}

Before stating the main result of this Section, let us start by describing the setting in which we will work on. 

\begin{setup} \label{set:setup}
Throughout this section $\psi$ will denote a bispecial  transformation of type $(m_-,m_+)$ between $(Y_-,H_-)$ and $(Y_+,H_+)$ as in Definition \ref{def:bispecial}. 
To shorten notation, we set $\mu_\pm:=m_\pm-1$, and 
 $L_\pm:=\pi_\pm^*H_\pm\in\Pic(W)$. 
\end{setup}

\begin{remark}\label{rem:equalities} 
Denoting by $F_\pm$ a fiber of $E_\pm \to Z_\pm$, we then have:
\begin{subequations}
\begin{align}
\label{eq:cr2}
(L_-)_{|F_-}&\sim\cO_{F_-},& (L_+)_{|F_+}&\sim\cO_{F_+},\\
\label{eq:cr3}
(L_-)_{|F_+}&\sim \cO_{F_+}(1), &(L_+)_{|F_-}&\sim\cO_{F_-}(1).
\end{align}
\end{subequations}
\end{remark}

\begin{theorem}\label{thm:gr1} Let $\psi$ be a bispecial transformation of type $(m_-,m_+)$ between $(Y_-,H_-)$ and $(Y_+,H_+)$ as in Setup \ref{set:setup}. There exists a unique B-type equalized smooth geometric realization $X$ of $\psi$ of criticality $3$ such that $Y_{\pm|Y_{\pm}}\simeq(1-m_\pm)H_\pm$.  
\end{theorem}

\begin{proof}
We will construct the variety $X$ as a birational modification of a $\P^1$-bundle on $W$, defined as
 $P:=\PP_W(L_-\oplus L_+)$, with natural projection $\pi:P \to W$ and tautological line bundle $H$. The proof will be divided in several steps. 

\medskip

\noindent{\em Step 1: Description of the variety $P$.} \par\medskip

We denote by $s_\pm:W \to P$ the two sections corresponding to the quotients $L_-\oplus L_+\to L_\pm$ and set $W_\pm:=s_\pm(W)$.
Then, by construction 
\begin{subequations}
\begin{align}\label{eq:PicP}
W_- &\sim H- \pi^*L_+, & W_+& \sim H- \pi^*L_-, \\
H|_{W_-} &\sim L_-,& H|_{W_+} &\sim L_+,\label{eq:restrH}\\
\intertext{so the normal bundles of $W_\pm$ in $P$ are given by}
W_-|_{W_-}&\sim L_- - L_+, & W_+|_{W_+}&\sim L_+ -L_-. \label{eq:normW}
\end{align}
\end{subequations}

Abusing notation, we denote by $E_{\pm}\subset W_\pm$ the images  $s_{\pm}(E_{\pm})\subset W_\pm$ of the exceptional loci of the projections $\pi_{\pm}$; note that we are not assuming $E_{\pm}$ to be irreducible. 
Moreover we denote by 
$\gamma$ a fiber of $\pi$. 

Note that there exists a natural equalized B-type $\C^*$-action on $P$, whose sink and source are, respectively, $W_-$ and $ W_+$. This is defined as the action of $\C^*$ on $P$ induced by the group of $(\cO_W$-module) automorphisms of $L_-\oplus L_+$ defined as the multiplication by $t\in\C^*$ in $L_-$ and as the identity in $L_+$.  

\medskip

\noindent{\em Step 2: Constructing two $\C^*$-equivariant small modifications $\varphi_\pm:P
 \dashrightarrow P_\pm$, 
with indeterminacy locus $E_\pm$.}\par\medskip

Each modification $\varphi_\pm$ will be obtained by composing the smooth blowup $p_\pm$ of $P$ along $E_\pm$ with a different smooth blowdown of the exceptional locus $E_\pm^\flat$ of $p_\pm$. 
We will show now how to construct $\varphi_-:P \dashrightarrow P_-$; the construction of $\varphi_+:P \dashrightarrow P_+$ is analogous.

Let $p_-:{P^\flat_-} \to P$ be the blowup of $P$ along $E_-\subset W_-$, with exceptional divisor $E_-^\flat:=\P(\cN^\vee)$, $\cN:=\cN_{E_-/P}$.
Since $E_-$ is contained in the fixed-point component $W_-$, the action of $\C^*$ can be lifted to $\cN$, providing the weight decomposition: 
\begin{equation}\label{eq:decompnormal}
\cN =\cN^0 \oplus \cN^- \simeq \cN_{E_-/W_-}\oplus (\cN_{W_-/P})|_{E_-}.
\end{equation}
Moreover, Equations (\ref{eq:decompnormal}, \ref{eq:normW}, \ref{eq:cr1}) imply that:
\begin{equation}\label{eq:pullback}\begin{aligned}
(\cN^0)^\vee\otimes\cO_{E_-}(E_-)&=\cO_{E_-}=\pi_-^*\cO_{Z_-}\\[2pt]
(\cN^-)^\vee\otimes\cO_{E_-}(E_-)&=\cO_{E_-}(\mu_-L_-)=\pi_-^*\cO_{Z_-}(\mu_-H_-)
\end{aligned}
\end{equation}
Then, denoting $B_-:
=\P_{Z_-}(\cO_{Z_-}\oplus \cO_{Z_-}(\mu_- H_-))$ we have a Cartesian square:
$$
\xymatrix@C=12mm@R=10mm{E_-^\flat\ar[r]^{\P^1}\ar[d]&E_-\ar[d]\\
B_-\ar[r]^{\P^1}&Z_-}
$$
whose maps are horizontal projective bundles.

On the other hand, since $E_-\subset W_-\subset P$ consists of $\C^*$-fixed points, it follows that the $\C^*$-action extends to $P_-^\flat$ so that $p_-$ is equivariant. Note that, since $E_-\subset W_-$ has codimension one, the strict transform via $p_-$ of $W_-$, which is the sink of the $\C^*$-action on $P_-^\flat$, is isomorphic to $W_-$; we continue to denote it by $W_-\subset P_-^\flat$. 

The action leaves $E_-^\flat$ invariant, and the weight decomposition (\ref{eq:decompnormal}) shows that $E_-^\flat$ contains two fixed-point components; they are the images of the two sections $\sigma_\pm:E_-\to E_-^\flat$ given by the two projections of $\cN^\vee$ onto its summands, $(\cN^0)^\vee,(\cN^-)^\vee$, respectively. The first one is contained in the sink of $P_-^\flat$, and the second is an inner fixed-point component of the action. 
The projection $E_-^\flat\to B_-$ is $\C^*$-equivariant, when one considers the (fiberwise) action of $\C^*$ on $B_-=\P_{Z_-}(\cO_{Z_-}\oplus \cO_{Z_-}(\mu_- H_-))$ whose fixed-point components are the images of the two sections corresponding to the summands of the subjacent vector bundle. We still denote these section by $\sigma_-$, $\sigma_+$ so that they fit into two commutative diagrams: 
$$
\xymatrix@C=12mm@R=10mm{E_-^\flat\ar[d]&E_-\ar[d]\ar[l]_{\sigma_\pm}\\
B_-&Z_-\ar[l]_{\sigma_\pm}}
$$
Note that any fiber $F_-^\flat$ of the projective bundle $E_-^\flat \to B_-$
is a section of $E_-^\flat \to E_-$ over its image $F_-$, which is a fiber of  $E_-\to Z_-$. By formulae (\ref{eq:cr3}, \ref{eq:normW}) we see that 
\begin{equation} \cN|_{F_-}\simeq \cO_{F_-}(-1) \oplus \cO_{F_-}(-1), \label{eq:norm}\end{equation} 
and in particular one may check that:
 $$
\cO_{E_-^\flat}(E_-^\flat)|_{F_-^\flat}\simeq \cO_{\P(\cO_{F_-}(1)^{\oplus 2})}(-1)_{|F_-^\flat}\simeq \cO_{F_-^\flat}(-1).$$
By Nakano Contractibility Criterion (\cite[Theorem~3.2.8]{BS}) there exists a proper map $q_-:P^\flat_- \to P_-$ which is the blowup of a smooth proper variety $P_-$ along a subvariety isomorphic to $B_-$. 
Note that $q_-|_{W_-}$ contracts $E_-^\flat\cap W_-$ to a variety isomorphic to $Z_-$; in other words, it is the blowdown map $W_- \to Y_-$.

Finally, since $\C^*$ is connected, 
$q_-$ is proper and satisfies that $q_{-*}\cO_{P^\flat_-}=\cO_{P_-}$, it follows by Blanchard's lemma (see \cite[Theorem~7.2.1]{Br17}, \cite[I.1]{Blan59}) that the $\C^*$-action on $P^\flat_-$ descends to an action on $P_-$ so that $q_-$ is $\C^*$-equivariant. 


The construction of $P_-^\flat$ and $P_-$ has been illustrated in Figure \ref{fig:flip}.

\begin{figure}[!h!] 
\includegraphics[width=12cm]{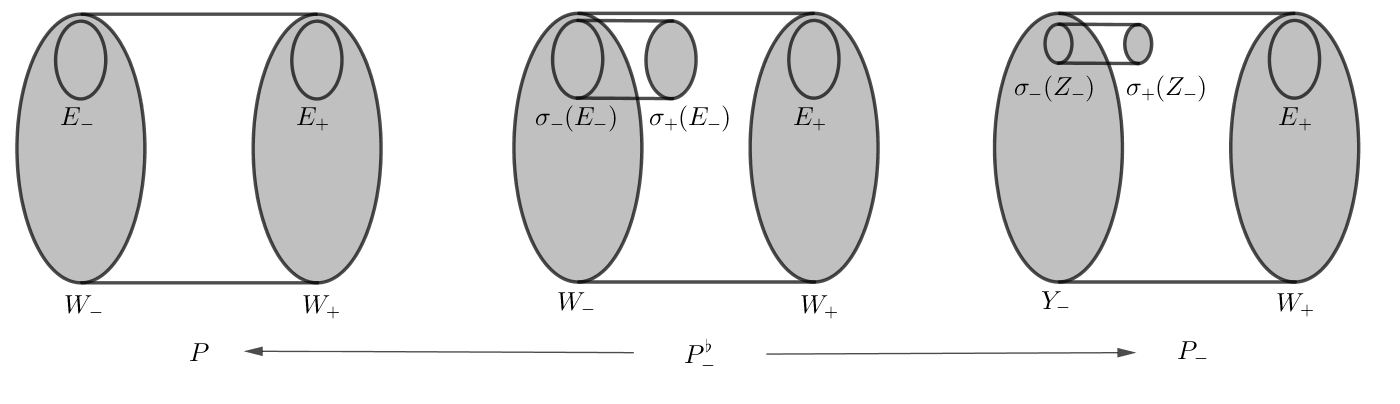}
\caption{The $\C^*$-equivariant small transformation $\varphi:P\dashrightarrow P_-$.\label{fig:flip}}
\end{figure}
\medskip

\noindent{\em Step 3: Description of the varieties $P_\pm$ and their orbit graphs.}\par\medskip 

By construction, the fixed-point components of $P_-$ are:
\begin{itemize}
\item the image via $q_-$ of the sink of $P_-^\flat$, which is isomorphic to $Y_-$;
\item the image via $q_-$ of $\sigma_+(E_-)$, which is isomorphic to $\sigma_+(Z_-)$;
\item the source $q_-(W_+)\simeq W_+$.
\end{itemize}
We denote them respectively by $Y_-,S_-,W_+\subset P_-$.
Denoting by $E_-^i$ the irreducible components of $E_-$, by $\delta_-^i$  the image in $P_-$ of a line in a fiber of $(E_-^i)^\flat \to E_-^i$, by $\gamma_-^i$ the strict transform of a fiber of $\pi$ meeting $E_-^i$, and by $\gamma$ the strict transform of a fiber of $\pi$ not meeting $E_-$, the action on $P_-$ has the following orbit graph:

\begin{equation}\label{eq:grP-}
\xygraph{
!{<0cm,0cm>;<3cm,0cm>:<0cm,1cm>::}
!{(0,0) }*+{\bullet_{Y_-}}="0"
!{(1,0) }*+{\bullet_{S_-^i}}="1"
!{(2,0) }*+{\bullet_{W_+}}="2"
"0"-@/^/@[blue]"1" _{\delta_-^i}
"0"-@/_0.5cm/@[blue]"2" _{\gamma}
"1"-@/^/@[blue]"2" _{\gamma_-^i}
} 
\end{equation}

In the small 
modification $P_+$, obtained blowing up $P$ along $E_+$ and blowing down the resulting variety $P_+^\flat$ to $P_+$, the induced action on $P_+$ has three fixed-point components $W_-$, $S_+$, $Y_+$ and orbit graph (for every component $E_+^j\subset E_j$):

\begin{equation}\label{eq:grP+}
\xygraph{
!{<0cm,0cm>;<3cm,0cm>:<0cm,1cm>::}
!{(0,0) }*+{\bullet_{W_-}}="0"
!{(1,0) }*+{\bullet_{S_+^j}}="1"
!{(2,0) }*+{\bullet_{Y_+}}="2"
"0"-@/^/@[blue]"1" _{\gamma_+^j}
"0"-@/_0.5cm/@[blue]"2" _{\gamma}
"1"-@/^/@[blue]"2" _{\delta_+^j}
} 
\end{equation}

\medskip

\noindent{\em Step 4: Merging $\varphi_\pm$ into a $\C^*$-equivariant small modification $\varphi:P\dashrightarrow X$ with indeterminacy locus $E_-\cup E_+$.} \par\medskip

The centers of the two birational transformations $\varphi_-:P \dashrightarrow P_-$ and $\varphi_+:P \dashrightarrow P_+$ described above are disjoint; therefore they can be performed simultaneously to define a small modification $\varphi:P\dashrightarrow X$. More concretely, we may write $P$ as the union of the two $\C^*$-invariant open subsets $P\setminus W_+$, $P\setminus W_-$; then the two birational transformations
$$
\varphi_-:P\setminus W_+ \dashrightarrow P_-\setminus W_+,\qquad \varphi_+:P\setminus W_- \dashrightarrow P_+\setminus W_-
$$
can be glued together into a map $\varphi: P\dashrightarrow X$ onto a smooth variety, which is obviously a small modification. Furthermore, this construction also shows that, since $\varphi_{\pm}$ are $\C^*$-equivariant, the action of $\C^*$ extends via $\varphi$ to $X$. By construction it has 
fixed-point components $Y_-, S_-, S_+$ and $Y_+$. 

Denoting by  $\gamma_\pm$ the strict transform of a fiber of $\pi$ meeting $E_\pm$ but not $E_\mp$ and by $\varepsilon^{ij}$ the strict transform of a fiber of $\pi$ meeting both $E_-^i$ and $E_+^j$  we have the following orbit graph:

\begin{equation}\label{eq:grX}
\xygraph{
!{<0cm,0cm>;<3cm,0cm>:<0cm,1cm>::}
!{(0,0) }*+{\bullet_{Y_-}}="0"
!{(1,0) }*+{\bullet_{S_-^i}}="1"
!{(2,0) }*+{\bullet_{S_+^j}}="2"
!{(3,0)}*+{\bullet_{Y_+}}="3"
"0"-@/^/@[blue]"1" _{\delta_-^i}
"2"-@/^/@[blue]"3" _{\delta_+^j}
"0"-@/_0.7cm/@[blue]"2" _{\gamma_+^j}
"1"-@/_0.7cm/@[blue]"3" _{\gamma_-^i}
"1"-@/^/@[blue]"2" _{\varepsilon^{ij}}
"0"-@/^0.8cm/@[blue]"3" ^{\gamma}
} 
\end{equation}
We are not assuming that $\varepsilon^{ij}$ exists for every $i,j$, but we claim that it does for some $i,j$. In fact,
by formulae (\ref{eq:cr1}, \ref{eq:cr2}, \ref{eq:cr3})  we have $E_-|_{F_+} \simeq \cO_{F_+}(m_-)$, so $E_- \cap E_+ \not = \emptyset$.
\par
\medskip
\noindent{\em Step 5: Projectivity of the variety $X$.}\par\medskip

Let us now show that $X$ is a projective variety, by finding an ample  line bundle on $X$. In view of  \cite[Lemma 2.2]{RW} and the Kleiman's criterion, a line bundle on $X$ is ample if and only if it has positive degree on the
classes of closures of orbits and on the classes of curves
contained in the fixed components. 
We first claim that the normal bundle of  $Y_\pm$ in $X$ is given by
\begin{equation}Y_\pm|_{Y_\pm}\sim -\mu_\pm H_\pm. \label{normYpm}\end{equation}

Let us prove the claim for $Y_-$, noting first that the line bundle associated to $Y_-|_{Y_-}$ equals the normal bundle $\cN_{Y_-|P_-}$ of $Y_-$ in $P_-$. 
In order to compute this bundle we first study $\cN_{W_-|P^\flat_-}$. This bundle fits into the following exact sequence of sheaves over $W_-$:
$$
0\to \cN_{W_-|P^\flat_-}\lra \cN_{W_-|P}\lra T_{E_-^\flat|E_-}(E_-^\flat)|_{\sigma_-(E_-)}\to 0
$$
Observing that $\cN_{W_-|P}\simeq L_--L_+$ by Equation (\ref{eq:normW}), and that
$$
T_{E_-^\flat|E_-}(E_-^\flat)|_{\sigma_-(E_-)}\simeq \cO_{E_-}(\mu_-L_--E_-)\simeq \cO_{E_-}(L_--L_+)
$$
by Equations (\ref{eq:decompnormal}), (\ref{eq:pullback}), (\ref{eq:cr1}),
 we get that 
$$
\cN_{W_-|P^\flat_-}\simeq \cO_{W_-}(L_--L_+-E_-)\simeq \cO_{W_-}(L_--L_+-m_-L_-+L_+)\simeq \cO_{W_-}(-\mu_-L_-).
$$
Since $\cO_{W_-}(-\mu_-L_-)\simeq \pi_-^*(-\mu_-H_-)$ we finally conclude that $$\cN_{Y_-|P_-}\simeq \cO_{Y_-}(-\mu_-H_-).$$

Since $B_-
=\P_{Z_-}(\cO_{Z_-}\oplus \cO_{Z_-}(\mu_- H_-))$, every curve in $B_\pm$ is numerically equivalent to a linear combination with nonnegative coefficients of the class of a curve in  $Z_\pm \subset Y_\pm$ and the class of the closure of a $1$-dimensional orbit joining $Y_-$ and  $S_-$. In particular, in order to show that a line bundle in $X$ is ample 
it is enough to test positivity  on the
classes of closures of orbits and on the classes of curves
contained in $Y_\pm$. By abuse of notation, we will denote again by $H$ the strict transform of the line bundle $H$ on $P$. Its intersection number with $\gamma, \gamma_\pm^i, \varepsilon^{ij}$, which are strict transforms of fibers of $P \to W$ is one, while $H \cdot \delta_\pm^i=0$.

By using Remark \ref{rem:Btypeequalized} we compute the intersection numbers of closures of $1$-dimensional  orbits with $Y_\pm$. Summing up, we have: 

\begin{center}
\begin{tabular}{C|C|C|C|C|C|C!}
& \delta_-^i& \gamma_-^i&   \delta_+^j& \gamma_+^j& \varepsilon^{ij} &\gamma  \\
\hline\hline
H  &0&1 & 0& 1& 1&1\\
\hline
Y_-&1& 0&0&1 & 0 &1\\
\hline
Y_+& 0& 1&1 &0&0&1\\
\hline
\end{tabular}
\end{center}
In particular any linear combination of $Y_\pm$ and $H$ with positive coefficients has positive degree on all the closures of  $1$-dimensional orbits.
Moreover
\[H|_{Y_\pm} \simeq H_\pm, \qquad Y_\pm|_{Y_\pm}\simeq -\mu_\pm H_\pm, \qquad Y_\pm|_{Y_\mp} \simeq \cO_{Y_\mp},\]
so any divisor $A$ in $X$ that is an integral multiple of a $\Q$-divisor of the form:
$$
H+\epsilon_-Y_-+\epsilon_+Y_+,\quad 0<\epsilon_-,\epsilon_+\ll 1,
$$
is ample on $X$. It is straightforward to check that the criticality of $(X,A)$ is three; this is also the criticality of $X$, by Lemma \ref{lem:AMFM}, because there exists a chain of closures of $1$-dimensional orbits joining the sink and the source of the form $\delta_-^i+ \varepsilon^{ij} + \delta^j_+$. \par
\medskip

\noindent{\em Step 6: Uniqueness of the geometric realization $X$.}\par\medskip  

We will show that the geometric realization $X$ of $\psi$ that we have constructed above is uniquely determined by the property of being of B-type, equalized and having criticality three. The proof of this Step will follow the line of argumentation of \cite[Section~8]{WORS1}, and \cite[Section~3]{WORS3}.  

Let us assume that $X'$,  together with an action of $\C^*$, is a geometric realization of $\psi$ satisfying the requirements of the theorem. 
Then we may consider the inner fixed-point sets $Y'_1,Y'_2$, that (see \cite[Lemma~8.5]{WORS1}) are isomorphic to $\Exc(\psi)$ and $\Exc(\psi^{-1})$. Then \cite[Theorem~3.1]{WORS3} tells us that the blowup of $X'$ along $\ol{X'^+(Y'_1)}$ and $\ol{X'^-(Y'_2)}$ admits a second divisorial contraction onto a smooth variety $P'$ with a B-type action with no inner fixed-point components, and whose sink and source are the blowup of $Y_-$ along $\Exc(\psi)$, and the blowup of $Y_+$ along $\Exc(\psi^{-1})$. 

It follows that these two blowups are isomorphic to the resolution $W$ of the bispecial transformation $\psi$, and that $P'$ is a $\P^1$-bundle over $W$. The sink and the source of the induced $\C^*$-action on $P'$, that we denote by $W'_\pm$, are two sections of $P'\to W$. 
Arguing as in the first part of Step 5, one may  compute the normal bundle of $Y_\pm$ in $X'$ out of the normal bundle of $W_\pm$ in $P'$; since the latter is, by hypothesis, $Y_{\pm|Y_{\pm}}\simeq (1-m_\pm)H_\pm$, one may check that $\cN_{W_\pm|P'}\simeq L_\pm-L_\mp$. It then follows that $P'\simeq\P_W(L_-\oplus L_+)\simeq P$, and so also $X'\simeq X$. 
\end{proof}

\section{Birational geometry of the geometric realization}\label{sec:cones}

In this section we describe some properties of the Nef and Movable cones of our geometric realization $X$ of a bispecial transformation $\psi$ between $(Y_-,H_-)$ and $(Y_+,H_+)$. Then we will apply this information to study contractions of $X$ in two particular situations (Sections \ref{ssec:qquadric}, \ref{ssec:21}). Our arguments provide information about a linear section of $\Nef(X)$ (precisely, the section generated by $Y_\pm$ and the strict transform of the tautological bundle $H$). For simplicity, we will 
work under the following assumptions: 

\begin{setup}\label{set:cones}
We consider two smooth projective 
varieties $Y_\pm$ of dimension $n$ such that $\Pic(Y_\pm) \simeq \Z$, 
and we denote by $H_\pm\in \Pic(Y_\pm)$ the ample generators of these groups. 
We consider a bispecial transformation $\psi$ between $(Y_-,H_-)$ and $(Y_+,H_+)$ of type $(m_-,m_+)$, and the corresponding B-type equalized geometric realization of criticality three $X$ satisfying that $Y_{\pm|Y_{\pm}}\simeq(1-m_\pm)H_\pm$ provided by Theorem \ref{thm:gr1}; as usual we set $\mu_{\pm}:=m_\pm-1$. We will denote by $[e_\pm]$  the classes of lines in the fibers of $E_\pm \to Z_\pm$, and set $[\ell_\pm]:=\pi_{\pm*}([e_\mp])$, which are classes of curves of $H_\pm$-degree $1$ in $Y_\pm$ (see Remark \ref{rem:classes}). We will consider them as classes in $X$ via the immersions of $Y_\pm$ as the sink and source of $X$, respectively.  
We will freely use the notation for auxiliary varieties, curves and divisors introduced in Section \ref{sec:geomreal}, particularly the classes of $\C^*$-invariant curves described in the graphs (\ref{eq:grP-}), (\ref{eq:grP+}) and (\ref{eq:grX}).
Furthermore, we will assume that the fundamental locus of $\psi$ is irreducible, so that, in particular, the inner components $S_\pm\subset X$ are irreducible (and so we may avoid the superscripts in the notation of the $\C^*$-invariant curves $\delta^i_\pm,\gamma^i_\pm,\varepsilon^{ij}$). 
\end{setup}

\begin{proposition}\label{prop:cones} Let $(Y_\pm,H_\pm)$, $\psi$ and $X$ be as in Setup  \ref{set:cones}, 
and let $P, P_-, P_+$ be the varieties constructed in the proof of Theorem \ref{thm:gr1}. Then the Nef cone and the Mori cone of these varieties are as described in Table \ref{tab:cones}. 

\renewcommand{\arraystretch}{1.55}
\begin{table}[h!!]
\begin{center}
\begin{tabular}{|C|C|C|}
\hline
\text{Variety} &\text{Nef cone}& \text{Mori cone} \\
\hline\hline
P  &\langle H, H-W_-, H-W_+\rangle& \langle [e_-], [e_+], [\gamma] \rangle\\
\hline
P_-, \mu_- >0& \langle H, H-W_+, H+\tfrac{Y_-}{\mu_-}, H+\tfrac{Y_-}{\mu_-} -W_+ \rangle& \langle [\delta_-], [e_+], [\ell_-],[\gamma_-] \rangle\\
\hline
P_-, \mu_- = 0&  \langle H, H-W_+, Y_- \rangle&\langle [e_+], [\delta_-], [\gamma_-] \rangle \\
\hline
P_+, \mu_+ >0& \langle H, H-W_-, H+\tfrac{Y_+}{\mu_+}, H+\tfrac{Y_+}{\mu_+} -W_- \rangle& \langle [\delta_+], [e_-], [\ell_+],[\gamma_+] \rangle\\
\hline
P_+, \mu_+ = 0&  \langle H, H-W_-, Y_+ \rangle&\langle  [e_-], [\delta_+], [\gamma_+] \rangle \\
\hline
X, \mu_-\mu_+>0 &\langle H, H+\tfrac{Y_-}{\mu_-}, H+\tfrac{Y_+}{\mu_+}, H+\tfrac{Y_-}{\mu_-} +\tfrac{Y_+}{\mu_+}\rangle &\langle [\delta_-], [\ell_+], [\ell_-],[\delta_+] \rangle \\
\hline
X, \mu_+=0 &  \langle H, H+\tfrac{Y_-}{\mu_-}, Y_+ \rangle&\langle [\delta_-], [\ell_-],[\delta_+] \rangle \\
\hline
\end{tabular}
\end{center}
\caption{Nef and Mori cones of $P,P_\pm,X$. \label{tab:cones}}
\end{table}
\renewcommand{\arraystretch}{1}
\end{proposition}

\begin{center}
\begin{figure}[!h]\caption{Cones}\label{fig:cones}\setlength\intextsep{0pt}
\subfloat[][The Mori cone $\NE(P)$]{
\begin{tikzpicture}[scale=1]
\draw (0,3)--(3,3)--(3,0)--(0,3);
\fill[black!90!white] (3,0) circle (0.8mm);
\fill[black!90!white] (3,3) circle (0.8mm);
\fill[black!90!white] (0,3) circle (0.8mm);
\node[anchor=north] at (-3,1) {};
\node[anchor=north] at (6,1) {};
\node[anchor=north] at (3,-0.1) {$e_-$};
\node[anchor=east] at (-0.1,3) {$e_+$};
\node[anchor=west] at (3,3) {$\gamma$};
\node[anchor=south west] at (0.7,0.7) {$H$};
\node[anchor=west] at (3,1.5) {$\pi^*L_-$};
\node[anchor=south] at (1.5,3) {$\pi^*L_+$};
\node[anchor=south] at (2.5,4) {};
\end{tikzpicture}}
\vspace{-0.5cm}
\subfloat[The Mori cone $\NE(P_+)$]{

\begin{tikzpicture}[scale=0.8]
\draw (2,0) --  (4,2)--(2,4)--(0,2)--(2,0);
\node[anchor=east] at (0,2) {$\ell_+$};
\fill[black!90!white] (2,0) circle (0.8mm);
\fill[black!90!white] (2,4) circle (0.8mm);
\fill[black!90!white] (4,2) circle (0.8mm);
\fill[black!90!white] (0,2) circle (0.8mm);
\node[anchor=west] at (4,2) {$\delta_+$};
\node[anchor=north] at (2,0) {$e_-$};
\node[anchor=south] at (2,4) {$\gamma_+$};
\node[anchor=south east] at (1,3) {$H+\tfrac{Y_+}{\mu_+} -W_-$};
\node[anchor=south west] at (3,3) {$H -W_-$};
\node[anchor=north west] at (3,1) {$H_{~} $};
\node[anchor=north east] at (1,1) {$H +\tfrac{Y_+}{\mu_+}$};

\draw (10,0)--(10,4)--(12,2)--(10,0);
\fill[black!90!white] (10,0) circle (0.8mm);
\fill[black!90!white] (12,2) circle (0.8mm);
\fill[black!90!white] (10,4) circle (0.8mm);
\fill[black!90!white] (10,2) circle (0.8mm);
\node[anchor=west] at (10,2) {$\ell_+$};
\node[anchor=south] at (10,4) {$\gamma_+$};
\node[anchor=north] at (10,0) {$e_-$};
\node[anchor=west] at (12,2) {$\delta_+$};
\node[anchor=south west] at (11,3) {$H -W_-$};
\node[anchor=north west] at (11,1) {$H_{~} $};
\node[anchor=east] at (10,2) {$Y_+$};
\node[anchor=south] at (6,4) {$(\mu_+>0)$};
\node[anchor=south] at (14,4) {$(\mu_+=0)$};
\node[anchor=south] at (2,6) {};
\end{tikzpicture}}
\vspace{-0.5cm}
\subfloat[The Mori cone $\NE(X)$, with the dual of $\Mov(X)$ shaded]{
\begin{tikzpicture}[scale=1]
\draw[fill=black!15]    (0,2)--(2,0) -- (2.66,2)-- (2.5,2.5)--(2,2.66) -- (0,2);
\draw[line width=0.0mm, fill=black!15]    (6,2)--(10,2) -- (9.666,2.666) -- (9,3) -- (6,2);
\draw (2,0)--(4,2)--(2,4)--(0,2)--(2,0);
\draw (2,0)--(2,4);
\draw (0,2)--(4,2);
\draw[line width=0.2mm] (4,2)--(1,3);
\draw[line width=0.2mm] (2,4)--(3,1);
\node[anchor=north] at (2,0) {$\ell_+$};
\fill[black!90!white] (2,0) circle (0.8mm);
\fill[black!90!white] (2,4) circle (0.8mm);
\fill[black!90!white] (4,2) circle (0.8mm);
\fill[black!90!white] (0,2) circle (0.8mm);
\fill[black!90!white] (2,2) circle (0.8mm);
\fill[black!90!white] (2,2.66) circle (0.8mm);
\fill[black!90!white] (2.66,2) circle (0.8mm);
\fill[black!90!white] (2.5,2.5) circle (0.8mm);
\node[anchor=south east] at (1,0.4) {$H + \tfrac{Y_-}{\mu_-} + \tfrac{Y_+}{\mu_+}$};
\node[anchor=north east] at (0.8,3.6) {$H +  \tfrac{Y_-}{\mu_-}$};
\node[anchor=north west] at (3.2,3.5) {$H$};
\node[anchor=south west] at (3.1,0.4) {$H + \tfrac{Y_+}{\mu_+}$};
\node[anchor=south] at (2,4) {$\delta_+$};
\node[anchor=south west] at (2.5,2.5) {$\gamma$};
\node[anchor=east] at (0,2) {$\ell_-$};
\node[anchor=north east] at (2,2) {$\varepsilon$};
\node[anchor=north east] at (2,3.3) {$\gamma_{-}$};
\node[anchor=north west] at (2.8,2) {$\gamma_{+}$};
\node[anchor=west] at (4.1,2) {$\delta_-$};
\draw (6,2)--(11,2)--(9,4)--(6,2);
\fill[black!90!white] (6,2) circle (0.8mm);
\fill[black!90!white] (11,2) circle (0.8mm);
\fill[black!90!white] (9,4) circle (0.8mm);
\fill[black!90!white] (9,2) circle (0.8mm);
\fill[black!90!white] (10,2) circle (0.8mm);
\fill[black!90!white] (9.666,2.666) circle (0.8mm);
\node[anchor=south west] at (9.56,2.66) {$\gamma$};
\draw (10,2)--(9,4);
\fill[black!90!white] (9,3.02) circle (0.8mm);
\node[anchor=south east] at (9.1,3.15) {$\gamma_{-}$};
\draw[line width=0.2mm] (11,2)--(8.15,3.45);
\draw[line width=0.2mm] (11,2)--(9,4);
\draw[line width=0.2mm] (9,2)--(9,4);
\node[anchor=north west] at (10.2,3.5) {$H$};
\node[anchor=north west] at (8.5,1.5) {$Y_+$};
\node[anchor=north east] at (7.2,3.6) {$H +  \tfrac{Y_-}{\mu_-}$};
\node[anchor=north] at (9,2) {$\ell_+$};
\node[anchor=north] at (10,2) {$\gamma_+$};
\node[anchor=south] at (9,4) {$\delta_+$};
\node[anchor=east] at (6,2) {$\ell_-$};
\node[anchor=west] at (11.1,2) {$\delta_-$};
\node[anchor=south] at (4,4) {$(\mu_+>0)$};
\node[anchor=south] at (11,4) {$(\mu_+=0)$};
\node[anchor=south] at (2,5.8) {};
\end{tikzpicture}}
\end{figure}
\end{center}

\begin{proof}[Sketch of Proof]
The idea of the proof is the same in all cases: we check that the listed line bundles are nef, by checking that they are nef when restricted to the sink and the source and on the closures of $1$-dimensional orbits. Then we check that each of them has degree zero on a two-dimensional face of the Mori cone.
\end{proof}

For the reader's convenience we have represented the Mori cones of the varieties $P$, $P_+$ ($P_-$ is analogous) and $X$ described in the statement in Figure  \ref{fig:cones}. 

 We can now use the information provided by Proposition \ref{prop:cones} to prove:

\begin{proposition}\label{prop:movx} Let $Y_\pm$, $\psi$ and $X$ be as in Setup \ref{set:cones}. If  $\mu_-\mu_+ >0$ then
\begin{multline*}\Mov(X) = \langle 
H-Y_-, H-Y_+,H+\tfrac{Y_-}{\mu_-}-Y_+,H+ \tfrac{Y_-}{\mu_-}+\tfrac{Y_+}{\mu_+},H- Y_-+\tfrac{Y_+}{\mu_+}\rangle.
\end{multline*}
If  $\mu_->0, \mu_+ =0$ then
\[\Mov(X) = \langle H-Y_-, H-Y_+,H+\tfrac{Y_-}{\mu_-}-Y_+,Y_+\rangle.
\]
\end{proposition}

\begin{proof}
If  $\mu_-\mu_+ >0$ then the Movable cone of $X$ is contained in 
$$
\cC =\langle [\gamma],[\gamma_-],[\ell_-],[\ell_+],[\gamma_+] \rangle^\vee,$$
since all  these curves move in codimension at least one. Computing intersection numbers we see that $\cC$ is  the cone  appearing in the statement. 
To prove that $\cC \subset \Mov(X)$  it is enough to check that $\cC$ has a chamber decomposition in which each chamber corresponds to the Nef cone of one of the birational modifications of $X$ described in Section \ref{sec:geomreal}. In fact these birational modifications are isomorphisms in codimension one, hence their Nef cones are contained in $\Mov(X)$.
We use the description of the cones given in Proposition \ref{prop:cones} to identify the chambers, as shown in Figure \ref{fig:mov} which represents a cross-section of $\Mov(X)$.
\begin{figure}[h]
\caption{The Movable cone $\Mov(X)$}\label{fig:mov}
\begin{center}
\scalebox{0.68}{
\begin{tikzpicture}[scale=0.5]
\fill[black!90!white] (0,0) circle (1mm); 
\node[anchor=north] at (0,0) {$Y_+$};
\fill[black!90!white] (16,0) circle (1mm); 
\node[anchor=north] at (16,0) {$Y_-$};
\fill[black!90!white] (6.02,17.29) circle (1mm); 
\node[anchor=south] at (6.02,17.29) {$H-Y_-$};
\fill[black!90!white] (10.02,17.29) circle (1mm); 
\node[anchor=south] at (10.02,17.29) {$H-Y_+$};
\fill[black!90!white] (2.8,8.05) circle (1mm); 
\node[anchor=east] at (2.9,8.3) {$H\!-\! Y_-\!+\!\tfrac{Y_+}{\mu_+}$};
\fill[black!90!white] (13.93,5.98) circle (1mm);
\node[anchor=west] at (14,6.5) {$H\!+\!\tfrac{Y_-}{\mu_-}\!-\!Y_+$};
\fill[black!90!white] (4.18,7.21) circle (1mm);
\node[anchor=west] at (3.1,5.7) {$\,\,H+\tfrac{Y_+}{\mu_+}$};
\fill[black!90!white] (9.39,4.03) circle (1mm);
\node[anchor=north] at (9.1,3.3) {$H+ \tfrac{Y_-}{\mu_-}+\tfrac{Y_+}{\mu_+}$};
\fill[black!90!white] (8.01,13.83) circle (1mm);
\node[anchor=north] at (8.01,13) {$H$};
\fill[black!90!white] (12.82,5.5) circle (1mm);
\node[anchor=east] at (13.5,4.2) {$H+\tfrac{Y_-}{\mu_-}$};

\draw (6.02,17.29) --  (12.82,5.5);
\draw (10.02,17.29) --  (4.18,7.21);
\draw (2.8,8.05) -- (6,17.29) -- (10,17.29) -- (13.93,6)-- (9.39,4.03)-- (2.8,8.05);
\draw[dashed] (2.8,8.05) --  (0,0);
\draw[dashed] (4.18,7.21) --  (0,0);
\draw[dashed] (9.39,4.03) --  (0,0);
\draw[dashed] (9.39,4.03) --  (16,0);
\draw[dashed] (13.93,5.98) --  (16,0);
\draw[dashed] (12.82,5.5) --  (16,0);
\draw[dashed] (16,0) --  (0,0); 

\node at (8,9) {$\Nef(X)$};
\node at (9.9,13.5) {$\Nef(P_-)$};
\node at (6.3,13.5) {$\Nef(P_+)$};
\node at (8,16.5) {$\,\Nef(P)$};
\node[anchor=south] at (1,16) {$\mu_+>0$};
\end{tikzpicture}

\hspace{-1 cm}
\begin{tikzpicture}[scale=0.5]
\fill[black!90!white] (0,0) circle (1mm); 
\node[anchor=north] at (0,0) {$Y_+$};
\fill[black!90!white] (16,0) circle (1mm); 
\node[anchor=north] at (16,0) {$Y_-$};
\fill[black!90!white] (6.02,17.29) circle (1mm); 
\node[anchor=south] at (6.02,17.29) {$H-Y_-$};
\fill[black!90!white] (10.02,17.29) circle (1mm); 
\node[anchor=south] at (10.02,17.29) {$H-Y_+$};
\fill[black!90!white] (13.93,5.98) circle (1mm);
\node[anchor=west] at (14,6) {$H\!+\!\tfrac{Y_-}{\mu_-}\!-\!Y_+$};
\fill[black!90!white] (8.01,13.83) circle (1mm);
\node[anchor=north] at (8.01,13) {$H$};
\fill[black!90!white] (12.82,5.5) circle (1mm);
\node[anchor=east] at (13.4,4) {$H+\tfrac{Y_-}{\mu_-}$};

\draw (6.02,17.29) --  (12.82,5.5);
\draw (10.02,17.29) --  (4.18,7.21);
\draw (0,0) -- (6,17.29) -- (10,17.29) -- (13.93,6)-- (9.39,4.03)-- (0,0);
\draw[dashed] (2.8,8.05) --  (0,0);
\draw (4.18,7.21) --  (0,0);
\draw[dashed] (13.93,5.98) --  (16,0);
\draw[dashed] (12.82,5.5) --  (16,0);
\draw[dashed] (16,0) --  (0,0); 

\node at (8,9) {$\Nef(X)$};
\node at (9.9,13.5) {$\Nef(P_-)$};
\node at (6.3,13.5) {$\Nef(P_+)$};
\node at (8,16.5) {$\,\Nef(P)$};
\node[anchor=south] at (1.5,16) {$\mu_+=0$};
\end{tikzpicture}}
\end{center}
\end{figure}
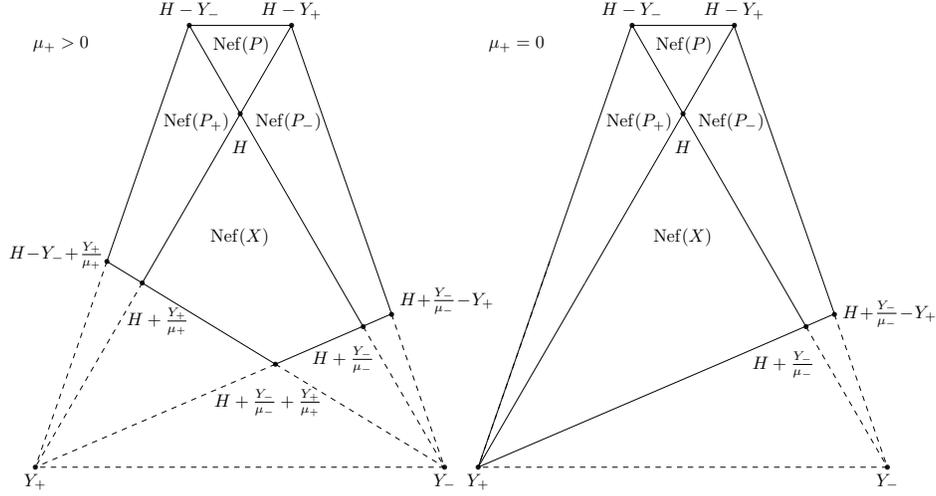

With similar arguments we can prove the statement also in the case $\mu_+>0, \mu_+=0$, where the chamber decomposition is shown in the right part of Figure \ref{fig:mov};
again we use Proposition \ref{prop:cones} to identify the chambers.
\end{proof}

\subsection{Bispecial transformations with $\mu_-\mu_+\neq 0$}
\label{ssec:qquadric}

In this Section we will consider a bispecial transformation $\psi$ between $(Y_-,H_-)$, $(Y_+,H_+)$ as in Setup \ref{set:cones}, and assume that $\mu_-\mu_+ >0$. We will consider the geometric realization $X$ of $\psi$ constructed in Theorem \ref{thm:gr1}, and show that it admits a $\C^*$-equivariant contraction onto a variety $X^\sharp$ of Picard number one, which is another geometric realization of $\psi$. We will finally study the smoothness of $X^\sharp$.

Along this section we will use the following notation:
$$
\mu:=\lcm(\mu_-,\mu_+),\qquad H':=\mu H+ \tfrac{\mu}{\mu_-}Y_-+\tfrac{\mu}{\mu_+}Y_+.
$$
By the description of the Mori cone $\NE(X)$ given in Proposition \ref{prop:cones}, $H'$ is nef and defines a $2$-dimensional face $\sigma$ of $\NE(X)$, whose extremal rays are generated by the numerical classes $[\ell_\pm]$. 
We start by proving the following statement:

\begin{lemma}\label{lem:Kneg} The face $\sigma$ is contained in the $K_X$-negative part of $\NE(X)$.
\end{lemma}

\begin{proof}
Denote by $r_\pm$ the codimension of $Z_\pm$ in $Y_\pm$. 
By adjunction we have
\[-K_P \cdot [s_+(e_+)] = (-K_{W_-} + W_-) \cdot [s_+(e_+)] = r_+\]
where the last equality is obtained using  (\ref{eq:normW}) and the blowup formula for the canonical bundle of $W_-$.
Denote by $p:P^\flat \to P$ the blowup along $E_-$ and $E_+$, by $q:P^\flat \to X$ the blowdown, by $\tilde e_+$ the strict transform of $s_+(e_+)$ in $P^\flat$. We can compute $K_{P^\flat}$ in two different ways using the blow-up formula for $p$ and $q$ and obtain:
\begin{equation}
p^*K_{P}- q^*K_{X} \sim (r_--2)E_-^{\flat} +(r_+-2)E_+^{\flat}, 
\end{equation}
 Then
\begin{equation*} 
-K_X \cdot [\ell_+] = q^*K_{X} \cdot [\tilde e_+] \ge p^*K_{P} \cdot [\tilde e_+] = -K_P \cdot [s_+(e_+)] = r_+.
\end{equation*}
A similar argument shows that $-K_X\cdot [\ell_-]>0$. 
\end{proof}

We may now construct the geometric realization $X^\sharp$ of Picard number one, and show that it is smooth if and only if $\psi$ is of type $(2,2)$ between two projective spaces. Note that in this case we have a complete classification of these realizations (see Section \ref{sec:BW3}). 

\begin{theorem}\label{thm:bispecial}
Let $(Y_\pm,H_\pm)$, $\psi$ and $X$ be as in Setup \ref{set:cones}, and assume that $\mu_-\mu_+\neq 0$. 
Then there exists a Mori contraction $\varphi_\sigma:X \to X^\sharp$, which contracts $Y_-$ and $Y_+$ to points. 

The variety $X^\sharp$ is smooth if and only if $Y_-$ and  $Y_+$ are projective spaces and $\psi$ is of type $(2,2)$. In this case there exists an ample $L \in Pic(X^\sharp)$ such that $(X^\sharp,L)$ is a bandwidth three variety with isolated extremal fixed points.
\end{theorem}

\begin{proof}
As a consequence of  Lemma \ref{lem:Kneg}, we have a Mori contraction  $\varphi_\sigma:X \to X^\sharp$ associated to the face $\sigma$, which contracts $Y_\pm$ to points. On the other hand, if $C$ is a curve not contained in $Y_\pm$ contracted  by $\varphi_\sigma$ we have $H'\cdot C=0$, from what if follows that $H\cdot C=Y_\pm\cdot C=0$, a contradiction. We conclude that $\varphi_\sigma$ is an isomorphism outside of $Y_\pm$.

Let $L \in \Pic(X^\sharp)$ be a line bundle such that $\varphi_\sigma^*L=H'$. Then the bandwidth of the induced action on  $(X^\sharp, L)$ can be computed intersecting $L$  with a general orbit (which is the image of $\gamma$), obtaining
\[L \cdot \gamma = \mu + \mu/\mu_- + \mu/\mu_+. \]
The divisor $Y_\pm$ is contracted to a smooth point if and only if $(Y_\pm, Y_\pm|_{Y_\pm}) \simeq (\P^{n-1}, \cO_{\P^{n-1}}(-1))$; by formula (\ref{normYpm}) the condition on the normal bundle is satisfied if and only if $\mu_\pm=1$, and in this case  the bandwidth of the $\C^*$-action on  $(X^\sharp, L)$ is three by Lemma \ref{lem:AMFM}. Summing up, we have shown that every bispecial transformation $\psi$ between two smooth projective varieties of Picard number one admits a geometric realization with isolated extremal fixed points, which is smooth if and only if $\psi$ is a bispecial Cremona transformation of type $(2,2)$. 
\end{proof}

\subsection{Bispecial transformations with $\mu_+= 0$}
\label{ssec:21}

In the situation of Setup \ref{set:cones}, we will consider now the case in which $\mu_+=0$; note that in this case we know that $\mu_->0$ (see Definition \ref{def:bispecial}). In analogy with the previous case, we will study contractions of the geometric realization of criticality three; we will show that $\psi$ corresponds to the birational map associated to a $\C^*$-action on a fibration over $\P^1$ of Picard number two, and study its smoothness.  

We will consider now the extremal ray $R:=\R_{\geq 0}[\ell_-]\subset \NE(X)$, and the $2$-dimensional face $\sigma\subset \NE(X)$ generated by $[\ell_-],[\ell_+]$. As in Lemma  \ref{lem:Kneg}, both faces are $K_X$-negative, and we have two associated Mori contractions $\varphi_R:X \to X^\sharp$, and $\varphi_\sigma:X \to T$; they are supported, respectively, by the nef line bundles $H':=\mu_-H+Y_-+Y_+$ and $Y_+$ (see Proposition \ref{prop:cones}). 

The map $\varphi_R$ contracts the curves in the class $[\ell_-]$, therefore it contracts $Y_-$ to a point. If it contracted a curve $C$ not contained in $Y_-$, we would have $H'\cdot C=0$. Since $Y_+$ and $H$ are nef, this implies that  $H\cdot C=Y_\pm\cdot C=0$, a contradiction.   
 It follows that $\varphi_R$ is the contraction of $Y_-$ to a point. In particular, $X^\sharp$ is smooth if and only if $(Y_-, Y_-|_{Y_-}) \simeq (\P^{n-1}, \cO_{\P^{n-1}}(-1))$ which (again by formula (\ref{normYpm})) is equivalent to say that $\mu_-=1$, that is if $m_-=2$.

On the other hand we consider the contraction $\varphi_\sigma:X\to T$, that factors via $\varphi_R:X \to X^\sharp$. The fact that $\varphi_\sigma$ is supported by $Y_+$, together with the equality $Y_+|_{Y_+}=\cO_{Y_+}$ (see Step 5 of Theorem \ref{thm:gr1}) implies that $\varphi_\sigma:X\to T$ is a fiber type contraction, and that $T$ is a curve. Since $Y_+\cdot \delta_+=1$, we obtain that $T=\varphi_\sigma(\delta_+)$ is necessarily rational, that is $T\simeq \P^1$.

Summing up we get the last statement of the paper:

 \begin{theorem}\label{thm:bispecial2}
Let $(Y_\pm,H_\pm)$, $\psi$ and $X$ be as in Setup \ref{set:cones}, and assume that $\mu_->0$, $\mu_+= 0$. 
Then there exists a Mori contraction $\varphi_R:X \to X^\sharp$, which is the contraction of $Y_-$ to a point. Moreover, $X^\sharp$ is a variety of Picard number two admitting a fiber type Mori contraction onto $\P^1$, that is smooth if and only if $Y_-$ is a projective space and $\psi$ is of type $(2,1)$.
\end{theorem}


\begin{thebibliography}{10}

\bibitem{Wlodarczyk-et-al}
Dan Abramovich, Kalle Karu, Kenji Matsuki, and Jaros{\l}aw W{\l}odarczyk.
\newblock Torification and factorization of birational maps.
\newblock {\em J. Amer. Math. Soc.}, 15(3):531--572, 2002.

\bibitem{B_R}
Lorenzo Barban and Eleonora~A. Romano.
\newblock Toric non-equalized flips associated to $\mathbb{C}^*$-actions.
\newblock {\em Preprint ArXiv:{\tt 2104.14442}}, 2021.

\bibitem{BS}
Mauro~C. Beltrametti and Andrew~J. Sommese.
\newblock {\em The adjunction theory of complex projective varieties},
  volume~16 of {\em De Gruyter Expositions in Mathematics}.
\newblock Walter de Gruyter \& Co., Berlin, 1995.

\bibitem{BB}
Andrzej Bia{\l}ynicki-Birula.
\newblock Some theorems on actions of algebraic groups.
\newblock {\em Ann. of Math. (2)}, 98:480--497, 1973.

\bibitem{BBS2}
Andrzej Bia{\l}ynicki-Birula and Joanna \'{S}wi{\c{e}}cicka.
\newblock Complete quotients by algebraic torus actions.
\newblock In {\em Group actions and vector fields ({V}ancouver, {B}.{C}.,
  1981)}, volume 956 of {\em Lecture Notes in Math.}, pages 10--22. Springer,
  Berlin, 1982.

\bibitem{Blan59}
Andr\'{e} Blanchard.
\newblock Sur les vari\'{e}t\'{e}s analytiques complexes.
\newblock {\em Ann. Sci. Ecole Norm. Sup. (3)}, 73:157--202, 1956.

\bibitem{Bourb}
Nicolas Bourbaki.
\newblock {\em \'{E}l\'ements de math\'ematique. {F}asc. {XXXIV}. {G}roupes et
  alg\`ebres de {L}ie. {C}hapitre {IV}: {G}roupes de {C}oxeter et syst\`emes de
  {T}its. {C}hapitre {V}: {G}roupes engendr\'es par des r\'eflexions.
  {C}hapitre {VI}: syst\`emes de racines}.
\newblock Actualit\'es Scientifiques et Industrielles, No. 1337. Hermann,
  Paris, 1968.

\bibitem{Br17}
Michel Brion.
\newblock Some structure theorems for algebraic groups.
\newblock In {\em Algebraic groups: structure and actions}, volume~94 of {\em
  Proc. Sympos. Pure Math.}, pages 53--126. Amer. Math. Soc., Providence, RI,
  2017.

\bibitem{BWW}
Jaros{\l}aw Buczy{\'n}ski, Jaros{\l}aw~A. Wi{\'s}niewski, and Andrzej Weber.
\newblock Algebraic torus actions on contact manifolds.
\newblock {\em To appear in J. Differ. Geom. Preprint ArXiv:{\tt 1802.05002}},
  2018.

\bibitem{CARRELL}
James~B. Carrell.
\newblock Torus actions and cohomology.
\newblock In {\em Algebraic quotients. {T}orus actions and cohomology. {T}he
  adjoint representation and the adjoint action}, volume 131 of {\em
  Encyclopaedia Math. Sci.}, pages 83--158. Springer, Berlin, 2002.

\bibitem{Coble}
A.~B. Coble.
\newblock Cremona transformations and applications to algebra, geometry, and
  modular functions.
\newblock {\em Bull. Amer. Math. Soc.}, 28(7):329--364, 1922.

\bibitem{Cre1}
Luigi Cremona.
\newblock Sulle trasformazioni geometriche delle figure piane. {N}ota {I}.
\newblock {\em Memorie dell'Accademia delle Scienze dell'Istituto di Bologna},
  tomo II(serie II):621--630, 1863.

\bibitem{Cre2}
Luigi Cremona.
\newblock Sulle trasformazioni geometriche delle figure piane. {N}ota {II}.
\newblock {\em Memorie dell'Accademia delle Scienze dell'Istituto di Bologna},
  tomo V(serie II):3--35, 1865.

\bibitem{ESB}
Lawrence Ein and Nicholas~I. Shepherd-Barron.
\newblock Some special {C}remona transformations.
\newblock {\em Amer. J. Math.}, 111(5):783--800, 1989.

\bibitem{IVERSEN}
Birger Iversen.
\newblock A fixed point formula for action of tori on algebraic varieties.
\newblock {\em Inv. Math.}, 16(3):229--236, 1972.

\bibitem{KKLV}
Friedrich Knop, Hanspeter Kraft, Domingo Luna, and Thierry Vust.
\newblock Local properties of algebraic group actions.
\newblock {\em Algebraische Transformationsgruppen und Invariantentheorie,
  Birkh{\"a}user Basel}, pages 63--75, 1989.

\bibitem{MMW}
Mateusz Micha{\l}ek, Leonid Monin, and Jaros{\l}aw~A. Wi\'sniewski.
\newblock Maximum likelihood degree and space of orbits of a
  $\mathbb{C}^*$-action.
\newblock {\em SIAM J. Appl. Algebra Geometry}, 1(5):60--85, 2021.

\bibitem{MFK}
David Mumford, John Fogarty, and France Kirwan.
\newblock {\em Geometric invariant theory}, volume~34 of {\em Ergebnisse der
  Mathematik und ihrer Grenzgebiete (2) [Results in Mathematics and Related
  Areas (2)]}.
\newblock Springer-Verlag, Berlin, third edition, 1994.

\bibitem{WORS2}
Gianluca Occhetta, Eleonora~A. Romano, Luis E.~Sol\'{a} Conde, and Jaros\l
  aw~A. Wi\'{s}niewski.
\newblock High rank torus actions on contact manifolds.
\newblock {\em Selecta Math. (N.S.)}, 27(1):Paper No. 10, 33, 2021.

\bibitem{WORS3}
Gianluca Occhetta, Eleonora~A. Romano, Luis~E. Sol\'{a}~Conde, and Jaros\l
  aw~A. Wi\'{s}niewski.
\newblock Small modifications of {M}ori dream spaces arising from
  {$\Bbb{C}^*$}-actions.
\newblock {\em Eur. J. Math.}, 8(3):1072--1104, 2022.

\bibitem{WORS1}
Gianluca Occhetta, Eleonora~A. Romano, Luis~E. Sol\'{a}~Conde, and Jaros\l
  aw~A. Wi\'{s}niewski.
\newblock Small bandwidth {$\Bbb C^*$}-actions and birational geometry.
\newblock {\em J. Algebraic Geom.}, 32(1):1--57, 2023.

\bibitem{PiRu}
Luc Pirio and Francesco Russo.
\newblock Varieties {$n$}-covered by curves of degree {$\delta$}.
\newblock {\em Comment. Math. Helv.}, 88(3):715--757, 2013.

\bibitem{RW}
Eleonora~A. Romano and Jaros\l aw~A. Wi\'{s}niewski.
\newblock Adjunction for varieties with a {$\Bbb{C}^*$} action.
\newblock {\em Transform. Groups}, 27(4):1431--1473, 2022.

\bibitem{Thaddeus}
Michael Thaddeus.
\newblock Complete collineations revisited.
\newblock {\em Math. Ann.}, 315(3):469--495, 1999.

\bibitem{Wlodarczyk}
Jaros{\l}aw W{\l}odarczyk.
\newblock Birational cobordisms and factorization of birational maps.
\newblock {\em J. Algebraic Geom.}, 9(3):425--449, 2000.

\end{thebibliography}
\end{document}